\documentclass[10pt, reqno, a4paper]{amsart}
\usepackage[colorlinks,citecolor=blue]{hyperref}
\usepackage{amsmath}
\usepackage{amscd}
\usepackage{amsthm}
\usepackage{amssymb}
\usepackage{latexsym}
\usepackage{euscript}
\usepackage{epsfig}
\usepackage{graphics}
\usepackage{psfrag}
\usepackage{array}
\usepackage{enumerate}
\usepackage{color}
\usepackage{wasysym}
\usepackage{graphicx}
\usepackage{amsfonts}
\usepackage{mathrsfs}
\usepackage{paralist}
\usepackage{bbm}

\newcommand{\Hmm}[1]{\leavevmode{\marginpar{\tiny%
$\hbox to 0mm{\hspace*{-0.5mm}$\leftarrow$\hss}%
\vcenter{\vrule depth 0.1mm height 0.1mm width \the\marginparwidth}%
\hbox to 0mm{\hss$\rightarrow$\hspace*{-0.5mm}}$\\\relax\raggedright
#1}}}

\numberwithin{equation}{section}
\newtheorem{Thm}{Theorem}[section]
\newtheorem{Rmk}[Thm]{Remark}
\newtheorem{Cor}[Thm]{Corollary}
\newtheorem{Lem}[Thm]{Lemma}

\newtheorem{Pro}[Thm]{Proposition}

\theoremstyle{definition}
\newtheorem{Def}[Thm]{Definition}

\newcommand{\bel}[1]{\begin{equation}\label{#1}}

\newcommand{\be}{\begin{equation}}

\newcommand{\ba}{\begin{eqnarray}}
\newcommand{\ea}{\end{eqnarray}}

\newcommand{\qe}{\end{equation}}

\newcommand{\R}{{\mathbb R}}

\newcommand{\Z}{{\mathbb Z}}
\newcommand{\C}{{\mathbb C}}

\newtheorem{thesis}{Thesis}
\newcommand{\btl}[1]{\begin{thesis}\label{#1}}
\newcommand{\et}{\end{thesis}}

\theoremstyle{theorem}

\theoremstyle{corollary}

\theoremstyle{lemma}

\theoremstyle{definition}

\theoremstyle{proof}

\theoremstyle{remark}

\usepackage{lipsum}

\newcommand\blfootnote[1]{%
  \begingroup
  \renewcommand\thefootnote{}\footnote{#1}%
  \addtocounter{footnote}{-1}%
  \endgroup
}
\date{}
\begin{document}

\title[Cheeger inequalities for magnetic Laplacians]{Frustration index
  and Cheeger inequalities for discrete and continuous magnetic
  Laplacians} \author{Carsten Lange} \address{Fachbereich f\"{u}r
  Mathematik und Informatik, Freie Universit\"{a}t Berlin, D-14195
  Berlin, Germany and Fakult\"at f\"ur Mathematik, Technische
  Universit\"at M\"unchen,  D-85748 Garching, Germany}
\email{clange@math.fu-berlin.de}

\author{Shiping Liu}
\address{Department of Mathematical Sciences, Durham University, DH1 3LE Durham, United Kingdom}
\email{shiping.liu@durham.ac.uk}

\author{Norbert Peyerimhoff}
\address{Department of Mathematical Sciences, Durham University, DH1 3LE Durham, United Kingdom}
\email{norbert.peyerimhoff@durham.ac.uk}

\author{Olaf Post}
\address{Fachbereich IV-Mathematik, Universit\"{a}t Trier, D-54286 Trier, Germany}
\email{olaf.post@uni-trier.de}
\blfootnote{2010 Mathematics Subject Classification 05C50 (35P15, 58J50).}
\begin{abstract}
  We discuss a Cheeger constant as a mixture of the frustration index
  and the expansion rate, and prove the related Cheeger inequalities and
  higher order Cheeger inequalities for graph Laplacians with cyclic
  signatures, discrete magnetic Laplacians on finite graphs and magnetic
  Laplacians on closed Riemannian manifolds. In this
  process, we develop spectral clustering algorithms for partially
  oriented graphs and multi-way spectral clustering algorithms
  via metrics in lens spaces and complex projective spaces. As a
  byproduct, we give a unified viewpoint of Harary's structural
  balance theory of signed graphs and the gauge invariance of magnetic
  potentials.

 \smallskip
\noindent \keywordsname:  frustration index; magnetic Laplacian; Cheeger's inequality; lens space; complex projective space; gauge transformation; magnetic potential; coarea formula; mixed graph.

\end{abstract}

\maketitle

\section{Introduction}

Cheeger's inequality is one of the most fundamental and important
estimates in spectral geometry. It was first proved by Cheeger for the
Laplace-Beltrami operator on a Riemannian manifold
\cite{Cheeger1970} and later extended to the setting of discrete
graphs, see e.g., \cite{Alon1986,AM1985,Dodziuk1984,BKW12},
demonstrating the close relationship between the spectrum and the geometry
of the underlying space. This inequality has a tremendous impact in
discrete and continuous theories and is an important intersection
point for interactions between both communities.
For example, it stimulated research in discrete mathematics such as
spectral clustering algorithms for data mining \cite{Luxburg07}, or
the construction of expander graphs \cite{HLW06}. Cheeger inequalities have also been considered on metric graphs,
see, e.g., \cite{Nic87} and, using a coarea formula in the proof,
\cite{Po09}. We recently witness several fruitful interactions in the other
direction: Lee, Oveis Gharan and Trevisan's higher order Cheeger
inequalities \cite{LOT2013} on finite graphs were used by Miclo
\cite{Miclo2013} to prove that hyperbounded, ergodic, and self-adjoint
Markov operators admit a spectral gap, solving a $40$-year-old conjecture
of Simon and H{\o}egh-Krohn \cite{SH1972}. For further developments, see
\cite{Liu13,Wang2014}.
Another example is an improved Cheeger's inequality for finite graphs
by Kwok et al. \cite{KLLGT2013}, which was subsequently used to
establish an optimal dimension-free upper bound of eigenvalue ratios
for weighted closed Riemannian manifolds with nonnegative Ricci
curvature \cite{Liu14} (see also \cite{LP14}). This answers open questions of
Funano and Shioya \cite{Funano2013,FS2013}.

Spectral theory of discrete and continuous magnetic Laplacians
attracted a lot of attention and literature on this subject developed
rapidly, see, e.g.,
\cite{Shigekawa87,Sunada93,LiebLoss1993,Shubin94,Erdos96,Paternain01,Shubin01,FLM08,ColindeVTT11,MT12,HinzTeplyaev13,Golenia14}.
Shigekawa proved the following comparison result in
\cite{Shigekawa87}: the least eigenvalue of the magnetic Laplacian on
a closed Riemannian manifold is bounded from above by the least
eigenvalue of a related Schr\"{o}dinger operator.  He also proved
Weyl's asymptotic formula for magnetic Laplacians. Paternain
\cite{Paternain01} obtained an upper bound of the least eigenvalue in
terms of the so-called harmonic value and Ma\~{n}\'{e}'s critical
value of the corresponding Lagrangian. On finite planar graphs, Lieb and Loss
\cite{LiebLoss1993} solved physically motivated extremality problems for eigenvalue expressions of the discrete magnetic Laplacian.



In this paper, we discuss a definition of Cheeger constants (Definitions \ref{def:Cheeger consant discrete},
\ref{def:Cheeger constant higher discrete} and \ref{def:Cheeger constant manifold}) reflecting the nontriviality of the magnetic potentials in terms of
the frustration index (see Definitions \ref{def:frustration index} and
\ref{def:frustration manifold}) and the global connectivity of the
underlying space. This definition works for both discrete and continuous
magnetic Laplacians, and graph Laplacians with $k$-cyclic signatures
($k \in {\mathbb N}$).
Recall that discrete magnetic Laplacians can be considered as graph
Laplacians with a $U(1)$-signature. We would like to point out that
our definition of Cheeger constants provides invariances under switching
operations (Definition \ref{def:switching}) or gauge transformations
(equation (\ref{eq:gaugeTrans})). Furthermore, we prove the
corresponding Cheeger inequalities and higher order Cheeger
inequalities (Theorems \ref{thm:CheegerInequality},
\ref{thm:CheegerInequalityU(1)}, \ref{thm:HigherOrder
  CheegerInequality}, \ref{thm:CheegerManifold}, and
\ref{thm:higherManifold}).  We notice that our Theorem
\ref{thm:CheegerInequalityU(1)}, the Cheeger inequality for discrete
magnetic Laplacian, overlaps with a Cheeger inequality of Bandeira,
Singer and Spielman \cite[Theorem 4.1]{BSS13} in the framework of
graph connection Laplacian \cite{SingerWu}. See Remark \ref{rmk:BSS}
for a more detailed explanation. It is known in physics that ``a
magnetic field raises the energy'' \cite{LiebLoss1993}. Roughly
speaking, our estimates tell us that a magnetic field raises the
energy via raising the frustration index. We focus on finite graphs
and compact Riemannian manifolds in this paper.

Cheeger inequalities are essentially coarea inequalities. In the
proof, we obtain in particular coarea inequalities related to the
frustration index on graphs as well as on manifolds
(Lemmata~\ref{lemma:coarea} and \ref{lemma:coarea manifold}).

In fact, we were led to our Cheeger constant definition by an
investigation of graph Laplacians with $k$-cyclic signatures, aiming
at extending a previous spectral interpretation \cite{AtayLiu14} of
Harary's structural balance theory \cite{Harary53,Harary59} for graphs
with $(\pm 1)$-signatures. It turns out that the Cheeger inequalities
for graph Laplacians with $k$-cyclic signatures and their proofs
provide spectral clustering algorithms for partially oriented graphs
(alternatively called mixed graphs without loops and multiple edges
\cite{HararyPalmer66,ZhangLi02,Ries07,SadeghiLauritzen14}), aiming at
detecting interesting substructures. A partially oriented graph may
contain both oriented and unoriented edges. In the proof of such
inequalities, we develop a random $k$-partition argument, which is
algorithmic (see Lemma \ref{lemma:key} and Proposition
\ref{prop:clusteringImodify}). Recall that, in the setting of $(\pm 1)$-signed
graphs (i.e., $k=2$), the eigenfunctions are real valued and a bipartition
of the underlying graph can be given naturally according to the sign
of the eigenfunction. But here we have complex valued
eigenfunctions. Hence we do not have any natural $k$-partitions. That
is why new ideas are needed. The generally non-symmetric graph
Laplacians of partially oriented graphs are hardly useful for the
purpose of spectral clustering. Our idea is to associate to a
partially oriented graph and a natural number $k \in {\mathbb N}$ an
unoriented graph with a special $k$-cyclic signature. We then perform
spectral clustering algorithms employing eigenfunctions of the graph
Laplacian with the associated signature. According to our Cheeger
constant definition, we can obtain interesting $k$-cyclic
substructures. See Section \ref{section:Direct graphs} for details.

To prove higher order Cheeger inequalities, we develop new multi-way spectral clustering
algorithms using metrics on \emph{lens spaces} and \emph{complex projective spaces}. This
provides a deeper understanding of earlier spectral clustering algorithms via metrics on
real projective spaces presented in \cite{Liu13} and \cite{AtayLiu14}.
These clustering algorithms were initially designed to find almost bipartite subgraphs
of a given graph, \cite{Liu13}, and then extended to find almost balanced subgraphs
of a signed graph, \cite{AtayLiu14}. While all operators studied in \cite{Liu13,AtayLiu14}
are bounded, we show that finding proper metrics for clustering is also useful for unbounded
operators: the spectral clustering algorithms via metrics on complex projective spaces
are crucial to prove the higher order Cheeger inequalities of the magnetic Laplacian on
a closed Riemannian manifold (Lemma \ref{lemma:localization manifold}).

\medskip
The paper is organized as follows. In Section \ref{section:basics}, we set up
notation for the discrete setting and recall basic spectral theory of
related graph operators. In Section \ref{section:Cheeger constant}, we define
the frustration index and the (multi-way) Cheeger constants. We prove the
corresponding Cheeger's inequality in Section \ref{section:Cheeger inequality}
and higher order Cheeger inequalities in Section \ref{section:higher Cheeger}.
In Section \ref{section:Direct graphs}, we discuss applications of Cheeger
inequalities for spectral clustering on partially oriented graphs. In Section
\ref{section:manifold case}, we extend the results developed on
discrete graphs to magnetic Laplacians on closed Riemannian manifolds.

\section{Notations and basic spectral theory}\label{section:basics}

Throughout the paper, $G=(V,E)$ denotes an undirected simple finite
graph on $N$ vertices with vertex set $V$ and edge set $E$. We denote
edges of $G$ by $\{u,v\}$, and $u \sim v$ means that $u \in V$ and $v
\in V$ are connected by an edge. For any subset $\widetilde V\subseteq V$,
let $\widetilde G=(\widetilde V,\widetilde E)$ be the subgraph of~$G$
induced by~$\widetilde V$, that is, an edge $\{ u,v\}$ of~$\widetilde G$
is an edge of~$G$ with $u,v\in \widetilde V$. We tacitly associate to every
edge $e=\{u,v\} \in E$ a positive symmetric weight $w_{uv} = w_{vu} = w_e$
and define the weighted degree $d_u$ of a vertex $u\in V$ by
$d_u:=\sum_{v,v\sim u}w_{uv}$. For a positive measure $\mu: V \to
{\mathbb R}^+$ on $V$, we define the \emph{maximal $\mu$-degree} of the graph
$G$ as
\begin{equation}
  d_{\mu}:=\max_{u\in V}\left\{\frac{\sum_{v,v\sim u}w_{uv}}{\mu(u)}\right\}
  = \max_{u\in V}\left\{\frac{d_u}{\mu(u)}\right\}.
\end{equation}
Henceforth we always consider weighted graphs, unless stated
otherwise, but refer to them simply as graphs. We denote by $e=(u,v)$
the oriented edge starting at $u$ and terminating at $v$, and by $\bar
e=(v,u)$ the oriented edge with the reversed orientation. Let
$E^{or}:=\{(u,v), (v,u)\mid \{u,v\}\in E\}$ be the set of all oriented
edges.

\begin{Def} Let $G$ be a graph and $\Gamma$ be a group. A \emph{signature} of
  $G$ is a map $s: E^{or} \to \Gamma$ such that
  \begin{equation}\label{eq:signatureKEY}
    s(\bar e)=s(e)^{-1},
  \end{equation}
  where $s(e)^{-1}$ is the inverse of $s(e)$ in $\Gamma$. The \emph{trivial
  signature} $s \equiv 1$, where $1$ stands for the identity element of $\Gamma$, is denoted by $s_1$. For an oriented edge
  $e=(u,v) \in E^{or}$, we will also write $s_{uv}:=s(e)$ for
  convenience.
\end{Def}

For $k \in {\mathbb N}$, we use the standard combinatorial notation
$[k] = \{1,2,\dots,k\}$. In this paper, we will restrict ourselves
to the case that the signature group $\Gamma$ is the cyclic group
$S_k^1:=\{\xi^{j} \mid j \in [k] \}$ of order $k$, generated by
the primitive $k$-th root of unity $\xi:=e^{2\pi i/k}\in \mathbb{C}$,
and the case that $\Gamma$ is the unitary group
$U(1) = \{ z \in {\mathbb C} \mid |z| = 1 \}$. The notation $S_k^1$
emphasizes the fact that the elements in $S_k^1$ lie on the unit circle.

We consider the following Laplacian $\Delta_{\mu}^s$ associated to the
weighted graph $(G,w)$ with signature $s: E^{or} \to \Gamma$ and
vertex measure $\mu: V \to {\mathbb R}^+$. For any function $f: V\to
\mathbb{C}$, and any vertex $u\in V$, we have
\begin{equation} \label{eq:Deltasmu}
\Delta_{\mu}^sf(u):=\frac{1}{\mu(u)}\sum_{v,v\sim u}w_{uv}(f(u)-s_{uv}f(v)).
\end{equation}
Note that the summation in \eqref{eq:Deltasmu} over the vertices $v$
adjacent to $u$ can also be understood as a summation over the
oriented edges $e=(u,v) \in E^{or}$, and the signature is evaluated at
$(u,v)$.

The Laplacian $\Delta_{\mu}^s$ has the following decomposition
$$\Delta_{\mu}^s=(D_{\mu})^{-1}(D-A^s)$$
where $D$ and $D_{\mu}$ are the diagonal matrices with $D_{uu}=d_u$
and $(D_{\mu})_{uu}=\mu(u)$ for all $u\in V$ while $A^s$ is the
(weighted) signed adjacency matrix with
$$
A^s_{uv} := \begin{cases}
			0,				& \text{$u=v$ or $\{u,v\}\not\in E$},\\
			w_{uv}s_{uv},	& \text{$\{u,v\}\in E$.}
\end{cases}
$$
When $\Gamma=S_k^1$, we call this operator the
graph Laplacian with the $k$-cyclic signature. When $\Gamma=U(1)$,
this is the discrete magnetic Laplacian studied in Sunada
\cite{Sunada93} (see also Shubin \cite{Shubin94}.)  By
(\ref{eq:signatureKEY}), the matrix $\Delta^s_{\mu}$ is Hermitian, and
hence all its eigenvalues are real which can be listed with
multiplicity as follows:
\begin{equation}
  0\leq \lambda_1(\Delta_{\mu}^s)\leq \lambda_2(\Delta_{\mu}^s)\leq\cdots\leq\lambda_N(\Delta_{\mu}^s)\leq 2d_{\mu}.
\end{equation}

For any two functions $f, g: V\to \mathbb{C}$, we define their inner
product as
\begin{equation}
  \langle f, g\rangle_{\mu}:=\sum_{u\in V}f(u)\overline{g(u)}\mu(u).
\end{equation}
It is easy to check that
\begin{equation}
  \langle \Delta_{\mu}^s f, g\rangle_{\mu}=\sum_{\{u,v\}\in E}w_{uv}(f(u)-s_{uv}f(v))(\overline{g(u)-s_{uv}g(v)}).
\end{equation}
Note that the right hand side of the above equality is well-defined since $\Gamma \subseteq U(1)$. The corresponding \emph{Rayleigh quotient}
$\mathcal{R}_{\mu}^s(f)$ of a function $f: V\to \mathbb{C}$ is
\begin{equation} \label{eq:Rayleigh}
  \mathcal{R}_{\mu}^s(f):=\frac{\sum_{\{u,v\}\in E}w_{uv}|f(u)-s_{uv}f(v)|^2}{\sum_{u\in V}|f(u)|^2\mu(u)}.
\end{equation}
The Courant-Fisher-Weyl min-max principle tells that, for any $n \in
[N]$,
\begin{equation}\label{eq:minmax}
  \lambda_n(\Delta_{\mu}^s)=\min_{\substack{\langle f_p, f_q\rangle_{\mu}=0,\\p,q \in [n], p \neq q}}\max_{f\in \mathrm{span}\{f_1,\ldots, f_n\}}\mathcal{R}_{\mu}^s(f),
\end{equation}
where $f_1,\dots,f_n,f \not\equiv 0$.

\begin{Rmk}
  In the case of a graph $G$ with measure $\mu_d(u):=d_u$ for all $u \in V$
  and signature group $\Gamma=U(1)$ or $\Gamma = S^1_k$, $k$ even, equation
  (\ref{eq:minmax}) implies the following relations between eigenvalues
  \begin{equation}
    2-\lambda_{N-k+1}(\Delta_{\mu_d}^{-s})=\lambda_k(\Delta_{\mu_d}^s).
  \end{equation}
  Here $-s$ is the signature obtained by taking the negative values of
  $s$ (as complex numbers). This generalizes
  \cite[Lemma 1]{AtayLiu14} where $\Gamma = S^1_2 = \{\pm 1\}$.
\end{Rmk}

There is a natural operation, called switching, acting on the
signatures \cite{Zaslavsky82,ZaslavskyMatrices}.

\begin{Def}\label{def:switching} Let $G$ be a graph with signature $s$. For any function
  $\tau: V\to \Gamma$ we can define a new signature
  $s^{\tau}:E^{or}\to \Gamma$ as follows:
  \begin{equation}
    s^{\tau}(e)=\tau(u)s(e)\tau(v)^{-1} \quad \forall\, e=(u,v)\in E^{or}.
  \end{equation}
  We call the function $\tau$ a \emph{switching function}.  The signature $s$
  and $s'$ are said to be \emph{switching equivalent} if there exists a
  switching function $\tau$ such that $s'=s^{\tau}$.
\end{Def}

One can check that switching is indeed an equivalence relation on the
set of signatures. An important invariant of the switching operation
is the spectrum of $\Delta_{\mu}^s$. In fact, it holds that (see
e.g. \cite{ZaslavskyMatrices})
\begin{equation}\label{unitary_equivalence_of_laplace}
  \Delta_{\mu}^{s^{\tau}}=D(\tau)\Delta_{\mu}^sD(\tau)^{-1},
\end{equation}
where $D(\tau)$ is the diagonal matrix with entries
$D(\tau)_{uu}=\tau(u)$. This means that $\Delta_{\mu}^{s^{\tau}}$ and
$\Delta_{\mu}^s$ are unitarily equivalent and have the same spectrum.
In particular, if the signature $s: E^{or}\to \Gamma$ is switching
equivalent to the trivial signature $s_1$, the operator
$\Delta_{\mu}^s$ is unitarily equivalent to the classical graph
Laplacian. In this case we have
$\lambda_1(\Delta_{\mu}^s)=0$. We will show in Section
\ref{section:Cheeger inequality} that this is the only case that the
first eigenvalue vanishes. Observe that on a tree, any signature is
switching equivalent to the trivial signature.

\begin{Rmk}
  The concept of switching is developed in the study of Harary's
  balance theory for signed graphs \cite{Harary53}, i.e. graphs with
  signatures $s:E^{or}\to S_2^1=\{+1,-1\}$, which we briefly review in
  the next section. The corresponding terminology in the magnetic
  theory is the gauge transformation, see, e.g.,
  \cite{ColindeVTT11,Shigekawa87}. Note that switching is an operation
  acting on the signatures $s_{uv}:=e^{i\alpha_{uv}}$, while the gauge
  transformation is acting on the magnetic potentials $\alpha_{uv}$,
  where $(u,v)\in E^{or}$. We will only use the terminology of the
  magnetic theory in the manifold case, see Section
  \ref{section:manifold case}. Switching equivalent signatures are
  called cohomologous weight functions in \cite{Sunada93}.
\end{Rmk}

\section{Frustration index and Cheeger
  constants}\label{section:Cheeger constant}

One of our motivations for introducing the Cheeger constants is
Harary's structural balance theory \cite{Harary53}. Let $G$ be a
finite graph with (possibly non-abelian) signature group $\Gamma$ and
signature $s: E^{or} \to \Gamma$, and
$\mathcal{C}$ be a cycle, which is a graph of the sequence $(u_1,u_2), (u_2,u_3), \cdots, (u_{l-1}, u_l), (u_l,u_1)$ of distinct edges. Then the signature of $\mathcal{C}$ is the conjugacy class of
the element
\begin{equation*}
s_{u_1u_2}s_{u_2,u_3}\cdots
s_{u_{l-1}u_l}s_{u_lu_1}\in \Gamma.
\end{equation*}
Note that the signature of a
cycle is switching invariant.

\begin{Def}
  A signature $s:E^{or}\to \Gamma$ is said to be \emph{balanced} if the
  signature of every cycle of $G$ is (the conjugacy class of the) identity
  element $1\in \Gamma$.
\end{Def}

For convenience, we will also say that the graph $G$ or a subgraph of
$G$ is balanced if the signature restricted on it is balanced.
Since the signature of a cycle is switching invariant, the property of being
balanced is also switching invariant. We have the following
characterization of being balanced using switching operations.

\begin{Pro}\label{Pro:switching lemma}{\rm (\cite[Corollary
    3.3]{Zaslavsky82})} A signature $s: E^{or}\to \Gamma$ is balanced
  if and only if it is switching equivalent to the trivial signature
  $s_{1}$.
\end{Pro}

\begin{Rmk}
  The concept of balance has been studied in the literature under
  various terminologies. For example, a balanced cycle is said to be
  satisfying Kirchhoff's Voltage Law in \cite{Gross74}. In
  \cite{ColindeVTT11}, the related concept to the signature of a cycle
  is the holonomy map. In magnetic theory, it is related to the
  magnetic flux \cite{LiebLoss1993}.
\end{Rmk}

We define the following frustration index to quantify how far a
signature on a subset is from being balanced.

\begin{Def}\label{def:frustration index}
  Let $G$ be a finite graph with signature $s$ and $V_1\subseteq V$ nonempty
  with induced subgraph~$(V_1,E_1)$. The \emph{frustration index} $\iota^s(V_1)$
  of $V_1$ is defined as
  \begin{align}\label{eq:frustration index}
    \iota^s(V_1):&=\min_{\tau: V_1\to
      \Gamma}\sum_{\{u,v\}\in E_1}w_{uv}|\tau(u)-s_{uv}\tau(v)| \\&=\min_{\tau: V_1\to
      \Gamma}\sum_{\{u,v\}\in E_1}w_{uv}|1-\tau(u)^{-1}s_{uv}\tau(v)|
  \end{align}
\end{Def}

A direct computation shows that the frustration index of a set is
switching invariant and, according to Proposition \ref{Pro:switching
  lemma}, we have
\begin{equation}\label{eq:frustrationproperty}
  \iota^s(V_1)=0 \Leftrightarrow \,\,\text{ the subgraph induced by } V_1 \,\text{ is balanced.}
\end{equation}
If $G$ is unweighted and $\Gamma=\{+1,-1\}$, then
\begin{equation}\label{eq:hararyindex}
  \iota^{s}(V)=2e_{min}^s(V),
\end{equation}
where $e_{min}^s(V)$ is the minimal number of edges that need to be
removed from $E$ in order to make $G=(V, E)$ balanced. The quantity $e_{min}^s(V)$ is exactly
the \emph{line index of balance} of Harary \cite{Harary59}. Having the work of Vannimenus and
Toulouse~\cite{VT77} in mind, Zaslavsky  suggested later the term "frustration index"
to Harary \cite{Zaslavsky14}.

We denote the \emph{boundary measure} of $V_1$ by
\begin{equation}
|E(V_1,V_1^c)|:=\sum_{u\in V_1}\sum_{v\in V_1^c}w_{uv},
\end{equation}
where $V_1^c$ is the complement of $V_1$ in $V$. The \emph{$\mu$-volume} of
$V_1$ is given by
\begin{equation}
\mathrm{vol}_{\mu}(V_1):=\sum_{u\in V_1}\mu(u).
\end{equation}

\begin{Def}\label{def:Cheeger consant discrete}
  Let $G$ be a finite graph with a signature $s$.  The \emph{Cheeger
  constant} $h_1^s(\mu)$ is defined as
  \begin{equation}\label{eq:Cheeger constant}
    h_1^s(\mu):=\min_{\emptyset\neq V_1\subseteq V}\phi_{\mu}^s(V_1),
  \end{equation}
  where
  \begin{equation}
    \phi^s_{\mu}(V_1):=\frac{\iota^s(V_1)+|E(V_1, V_1^c)|}{\mathrm{vol}_{\mu}(V_1)}.
  \end{equation}
\end{Def}

The choice of $V_1$ achieving the minimum in (\ref{eq:Cheeger
  constant}) can be viewed as a subset of vertices which balances the
two complementary goals of minimizing its frustration index and its
expansion, measured by the edges $E(V_1, V_1^c)$ connecting $V_1$ with
its complement.

A \emph{nontrivial $n$-subpartition} of $V$ is given by $n$ pairwise disjoint
nonempty subsets $V_1,\dots,V_n \subset V$ and a \emph{nontrivial
$n$-partition} additionally satisfies $\bigcup_{p \in [n]} V_p=V$. We
abbreviate a nontrivial $n$-(sub)partition $\{V_1, \ldots , V_n\}$ by
$\{V_p\}_{[n]}$. In the spirit of Miclo \cite{Miclo2008}, we define the
multi-way Cheeger constants as follows.

\begin{Def}\label{def:Cheeger constant higher discrete}
  Let $G$ be a finite graph with a signature $s$.  The \emph{$n$-way Cheeger
  constant} $h_n^s(\mu)$ of $G$ is defined as
  \begin{equation}\label{eq:HigherCheeger constant}
    h_n^s(\mu):=\min_{\{V_p\}_{[n]}}\  \max_{p \in [n]}\phi_{\mu}^s(V_p),
  \end{equation}
  where the minimum is taken over all nontrivial
  $n$-subpartitions $\{V_p\}_{[n]}$ of $V$.
\end{Def}

Observe that the $n$-way Cheeger constant of a graph $G$ is monotone
with respect to $n$, that is, $h_n^s(\mu)\leq h_{n+1}^s(\mu)$.

Using (\ref{eq:frustrationproperty}) and the fact that the frustration
index is switching invariant, we obtain the following properties of
the Cheeger constants.

\begin{Pro}\label{pro:Cheeger constant}
  The n-way Cheeger constants $h_n^s(\mu)$ of a graph $G$ are
  switching invariant. Moreover, $h_n^s(\mu)=0$ if and only if $G$
  consists of at least $n$ connected components and at least $n$
  of them are balanced.
\end{Pro}

If $s_b: E^{or}\to \Gamma$ denotes a balanced signature, then
$h_{1}^{s_b}(\mu)=0$ becomes trivial and
\begin{equation}
  h_2^{s_b}(\mu)=\min_{\{V_1, V_2\}}\ \max_{p\in [2]}\ \phi_{\mu}^{s_b}(V_p)
		= \min_{\substack{\emptyset\neq V_1\subseteq V\\\mathrm{vol}_{\mu}(V_1)\leq \frac{1}{2}\mathrm{vol}_{\mu}(V)}} 		
			\frac{|E(V_1, V_1^c)|}{\mathrm{vol}_{\mu}(V_1)},
\end{equation}
that is, $h_2^{s_b}(\mu)$ reduces to the classical Cheeger constant.

\begin{Rmk}
  Due to equation (\ref{eq:hararyindex}), the $n$-way Cheeger
  constant in (\ref{eq:HigherCheeger constant}) reduces to the signed
  Cheeger constant introduced on signed graphs \cite{AtayLiu14} with
  signature group $\Gamma=\{+1,-1\}$. We mention that the signed
  Cheeger constant in \cite{AtayLiu14} is a unification of the
  classical Cheeger constant, the non-bipartiteness parameter in
  \cite{DR1994}, the bipartiteness ratio in \cite{Trevisan2012}, and
  the dual Cheeger constant in \cite{BJ}.
\end{Rmk}

For $n \in [N]$ and any signature $s: E^{or}\to \Gamma$, we observe
\begin{equation}\label{eq:3.11}
  h_n^{s_b}(\mu)\leq h_n^{s}(\mu).
\end{equation}
In fact, let $\{\widetilde{V}_p\}_{[n]}$ be the nontrivial $n$-subpartition of $V$ that achieves $h_n^s(\mu)$, i.e. $h_n^s(\mu)=\max_{p\in [n]}\phi_\mu^s(\widetilde{V}_p)$, we have $\phi_\mu^{s_b}(\widetilde{V}_p)\leq \phi_\mu^s(\widetilde{V}_p)$ since $\iota^{s_b}(\widetilde{V}_p)=0\leq \iota^{s_b}(\widetilde{V}_p)$. Hence, (\ref{eq:3.11}) follows by Definition \ref{def:Cheeger constant higher discrete}.
The inequality (\ref{eq:3.11}) is similar, in spirit, with Kato's inequality for noncompact
spaces \cite[Lemma 1.2, Corollary 1.3]{DodziukMathai06}
(alternatively, also called the diamagnetic inequality for both
compact and noncompact spaces in \cite{LiebLoss1993}) where the bottom
of the spectrum increases when a balanced signature is replaced by an
unbalanced signature.

For $n=1$ we have the following result. Recalling $h_1^{s_b}(\mu)=0$,
Proposition \ref{prop:spielm} tells us that this change of the first
Cheeger constant (by choosing an unbalanced signature) can be quite
large.

\begin{Pro} \label{prop:spielm} Let $G$ be an unweighted connected
  finite $d$-regular graph and $M = \max_{v \in V} \mu(v)$. Then, for
  every $k \ge 2$, there exists a $k$-cyclic signature $s_0: E^{or}\to
  S_k^1$ such that
  \begin{equation}\label{eq:existLargeCheeger}
    h_1^{s_0}(\mu)\geq \frac{d-2\sqrt{d-1}}{2M}.
  \end{equation}
\end{Pro}

\begin{proof}
  Extending a result of~\cite{MSS}, it is shown in~\cite[Theorem 2]{LPV14}
  that there exists a $k$-cyclic signature $s_0$ such that the maximal
  eigenvalue of the matrix $A^{s_0}$ is no greater than $2\sqrt{d-1}$. The
  estimate (\ref{eq:existLargeCheeger}) is then an immediate consequence of
  this result, combined with Cheeger's inequality (\ref{eq:Cheeger inequality}),
  given at the beginning of the next section.
\end{proof}

\section{Cheeger's inequality}\label{section:Cheeger inequality}

In this section, we prove Cheeger's inequality relating
$\lambda_1(\Delta_{\mu}^s)$ to the first Cheeger constant
$h_1^{s}(\mu)$ for graph Laplacians with cyclic signatures (Theorem
\ref{thm:CheegerInequality}) and for discrete magnetic Laplacians
(Theorem \ref{thm:CheegerInequalityU(1)}).

\begin{Thm}\label{thm:CheegerInequality}
  Let $G$ be a finite graph with signature $s: E^{or}\to
  S_k^1$. Then we have
  \begin{equation}\label{eq:Cheeger inequality}
    \frac{1}{2}\lambda_1(\Delta_{\mu}^s)\leq h_1^s(\mu)\leq 2\sqrt{2d_{\mu}\lambda_1(\Delta_{\mu}^s)}.
  \end{equation}
\end{Thm}

We start with preparations for the proof of Theorem
\ref{thm:CheegerInequality}. Let $B_r(0):=\{z\in \mathbb{C}\mid |z|<
r\}$ be the open disk in $\mathbb{C}$ with center $0$ and radius
$r$. For $\theta\in [0, 2\pi)$ and $k \in {\mathbb N}$, we define the
following $k$ disjoint sectorial regions
\begin{equation}\label{eq:notationQ}
  Q_{j}^{\theta}:=\left\{re^{i\alpha}\in \overline{B_1(0)}\left|\,\, r\in (0,1], \alpha\in \left[\theta+\frac{2\pi j}{k}, \theta+\frac{2\pi (j+1)}{k}\right)\right.\right\},
\end{equation}
where $j=0,1,\ldots, k-1$. Then for any $t\in (0,1]$, we define the function
$Y_{t, \theta}: \overline{B_1(0)}\to \mathbb{C}$ as
\begin{equation}
  Y_{t, \theta}(z):=\left\{
    \begin{array}{ll}
      \xi^j, & \hbox{if $z\in Q_j^{\theta}\setminus B_{t}(0)$,} \\
      0, & \hbox{if $z\in B_{t}(0)$,}
    \end{array}
  \right.
\end{equation}
where $\xi$ denotes the $k$-th primitve root of unity.

\smallskip

The following lemma plays a key role.

\begin{Lem}\label{lemma:key}
  For any two points $z_1, z_2\in \overline{B_1(0)}$, we have
  \begin{equation}\label{eq:keylemma}
    \frac{1}{2\pi}\int_0^{2\pi}\int_0^1 \left| Y_{\sqrt{t}, \theta}(z_1)-
      Y_{\sqrt{t}, \theta}(z_2) \right|\, dt\, d\theta\, \leq \,
    2\, |z_1-z_2|\, (|z_1|+|z_2|).
  \end{equation}
\end{Lem}

\begin{proof} W.l.o.g., we can assume that $|z_1| \ge |z_2|$ with
  $z_1\in Q_{j_1}^{\theta}$ and $z_2\in Q_{j_2}^{\theta}$. Then we
  have
  \begin{equation}
    |Y_{\sqrt{t}, \theta}(z_1)-Y_{\sqrt{t}, \theta}(z_2)|=\left\{
      \begin{array}{ll}
        |\xi^{j_1}-\xi^{j_2}|, & \hbox{if $\sqrt{t}\leq |z_2|$,} \\
        1, & \hbox{if $|z_2|<\sqrt{t}\leq |z_1|$,} \\
        0, & \hbox{if $|z_1|<\sqrt{t}$.}
      \end{array}
    \right.
  \end{equation}
  Hence,
  \begin{equation}
    \int_0^1\left| Y_{\sqrt{t}, \theta}(z_1)-Y_{\sqrt{t}, \theta}(z_2)\right|\, dt=|\xi^{j_1}-\xi^{j_2}|\cdot|z_2|^2+(|z_1|^2-|z_2|^2).
  \end{equation}
  Let $\alpha_{z_1z_2}\in [0, \pi]$ be the angle between the two rays
  joining $z_1, z_2$ to the origin. If $2\pi l/k\leq \alpha_{z_1z_2}<
  2\pi (l+1)/k$ for some integer $0\leq l<k/2$, the term
  $|\xi^{j_1}-\xi^{j_2}|$ is equal to either $|1-\xi^l|$ or
  $|1-\xi^{l+1}|$, hence we calculate
  \begin{equation*}
  \frac{1}{2\pi}\int_0^{2\pi}\int_0^1 \left| Y_{\sqrt{t},
        \theta}(z_1)- Y_{\sqrt{t}, \theta}(z_2) \right|\, dt\, d\theta =
  \phantom{\frac{k\left(\alpha_{z_1z_2}-2\pi l/k\right)}{2\pi}
    \left(|1-\xi^{l+1}|\cdot|z_2|^2 + |z_1|^2 - |z_2|^2 \right)}
  \end{equation*}
  \begin{eqnarray*}
  &=& \left(\frac{k\alpha_{z_1z_2}}{2\pi}-l\right)
    \left(|1-\xi^{l+1}|\cdot|z_2|^2 + |z_1|^2 - |z_2|^2 \right)\\
  && +\left(l+1-\frac{k\alpha_{z_1z_2}}{2\pi}\right)\cdot
    \left(|1-\xi^{l}|\cdot|z_2|^2+|z_1|^2-|z_2|^2\right)\\
  &\le& 2|1-\xi^l|\cdot|z_2|^2\, +\, \left(|z_1|^2-|z_2|^2\right),
  \end{eqnarray*}
  where we used $|1-\xi^{l+1}|\leq |1-\xi|+|1-\xi^l|\leq 2|1-\xi^l|$.
  Observe that we have
  \begin{equation}\label{eq:keyest1}
    |z_1-z_2|\geq \left|\frac{z_1}{|z_1|}|z_2|-z_2\right|\geq
    |z_2|\cdot|1-\xi^l|
  \end{equation}
  and
  \begin{equation}\label{eq:keyest2}
    |z_1|^2-|z_2|^2=(|z_1|-|z_2|)\cdot(|z_1|+|z_2|)\leq
    |z_1-z_2|\cdot(|z_1|+|z_2|).
  \end{equation}
  Therefore, we obtain
  \begin{equation}
    \frac{1}{2\pi}\int_0^{2\pi}\int_0^1 \left|Y_{\sqrt{t}, \theta}(z_1)-
    Y_{\sqrt{t}, \theta}(z_2)\right|\, dt\, d\theta \leq
    2|z_1-z_2|\cdot|z_2|+|z_1-z_2|\cdot(|z_1|+|z_2|),
  \end{equation}
  which implies (\ref{eq:keylemma}).
\end{proof}

Lemma \ref{lemma:key} can be considered as an extension of \cite[Lemma
5]{AtayLiu14} and \cite[Section 3.2]{Trevisan2012}. The novel point
here is that we introduce an extra degree of randomness in the
argument of $z$ in order to handle the difficulty caused by cyclic
signatures. Actually, this provides a random $k$-partition
parametrized by an angle $\theta$, which will be discussed further in
Section \ref{section:Direct graphs}. This lemma is a version of a
coarea inequality, which becomes transparent from the following direct
consequence.

For any non-zero function $f: V\to \mathbb{C}$ defined on the vertices
of a graph $G$ and any $t\in [0, \max_{u\in V}|f(u)|]$, we define the
following non-empty subset of $V$:
\begin{equation}\label{eq:notationVft}
  V^f(t):=\{u\in V \, \mid \, t \leq |f(u)| \}.
\end{equation}

\begin{Lem}[Coarea inequality]\label{lemma:coarea}
  Let $s: E^{or}\to S_k^1$ be a signature of $G$. For any
  function $f:V\to \mathbb{C}$ with $\max_{u\in V}|f(u)|=1$, we have
  \begin{multline}\label{eq:coarea inequality}
    \int_0^1 \iota^s\left(V^f(\sqrt{t})\right)\, + \,
      \left|E\left(V^f(\sqrt{t}), (V^f(\sqrt{t}))^c
        \right)\right| \, dt \\ \leq 2\sum_{\{u,v\}\in
      E}w_{uv}\,
    \left|f(u)-s_{uv}f(v)\right|\cdot\left(|f(u)|+|f(v)|\right).
  \end{multline}
\end{Lem}

\begin{proof}
  First observe that
  \begin{multline}
    \frac{1}{2\pi}\int_0^{2\pi}\int_0^1\sum_{\{u,v\}\in E}w_{uv}\,
    \left| Y_{\sqrt{t}, \theta}(f(u))-
      s_{uv}Y_{\sqrt{t}, \theta}(f(v))\right|\, dt\, d\theta\\
    \geq \int_0^1 \iota^s\left(V^f(\sqrt{t})\right)\, +\,
      \left|E(V^f(\sqrt{t}),
        (V^f(\sqrt{t}))^c)\right|\, dt.\label{eq:estimatemodify}
  \end{multline}
In fact, the summation in the integrand of the LHS of the above inequality can be split into two parts: The summation over edges connecting two vertices from $V^f(\sqrt{t})$ and $V^f(\sqrt{t})^c$, respectively. This part equals to $\left|E(V^f(\sqrt{t}), (V^f(\sqrt{t}))^c)\right|$; The summation over edges connecting two vertices from $V^f(\sqrt{t})$. This part is bounded from below by $\iota^s\left(V^f(\sqrt{t})\right)$ by Definition \ref{def:frustration index}.
  
Notice further that
  \begin{equation}
    s_{uv}Y_{\sqrt{t},\theta}(f(v))=Y_{\sqrt{t},\theta}(s_{uv}f(v)),
  \end{equation}
the inequality (\ref{eq:coarea inequality}) follows directly from
  Lemma \ref{lemma:key}.
\end{proof}

The Coarea Inequality is particularly useful to prove Lemma
\ref{lemma:clusteringI}.

\begin{Lem}\label{lemma:clusteringI}
  Let $s: E^{or}\to S_k^1$ be a signature of $G$ and $f:V\to
  \mathbb{C}$ be a nonzero function. Then there exists $t'\in [0,
  \max_{u\in V}|f(u)|^2]$ such that
  \begin{equation} \label{eq:phisR}
    \phi^s_{\mu}(V^f(\sqrt{t'}))\leq 2\sqrt{2d_{\mu}\mathcal{R}_{\mu}^s(f)},
  \end{equation}
  where $\mathcal{R}_{\mu}^s(f)$ was defined in \eqref{eq:Rayleigh}.
\end{Lem}

\begin{proof}
  Since $f$ is non-zero, we may assume (after rescaling) that
  $\max_{u\in V}|f(u)|=1$. Moreover,
  \begin{equation}
    |Y_{\sqrt{t},\theta}(f(u))|=\left\{
      \begin{array}{ll}
        1, & \hbox{if $|f(u)|\geq \sqrt{t}$,} \\
        0, & \hbox{otherwise,}
      \end{array}
    \right.
  \end{equation}
  implies
  \begin{equation}\label{eq:area formula}
    \int_0^1\mathrm{vol}_{\mu}(V^f(\sqrt{t}))\, dt = \int_0^1\sum_{u\in V}\left|Y_{\sqrt{t},\theta}(f(u))\right|\, \mu(u)\, dt = \sum_{u\in V}|f(u)|^2\mu(u).
  \end{equation}
  Now we consider the quotient
  \begin{equation}
    I:=\frac{\int_0^1 \iota^s(V^f(\sqrt{t}))+\left|E(V^f(\sqrt{t}), (V^f(\sqrt{t}))^c)\right|\, dt}{\int_0^1\mathrm{vol}_{\mu}(V^f(\sqrt{t}))dt}.
  \end{equation}
  Therefore, there exists $t'\in [0,1]$ such that
  \begin{equation}\label{eq:lowerI}
    I\geq \phi^s_{\mu}(V^f(\sqrt{t'})).
  \end{equation}
  On the other hand, Lemma \ref{lemma:coarea}, (\ref{eq:area
    formula}), and the Cauchy-Schwarz inequality imply
  \begin{align*}
    I&\leq \frac{2\sum_{\{u,v\}\in E}w_{uv}|f(u)-s_{uv}f(v)|\cdot(|f(u)|+|f(v)|)}{\sum_{u\in V}|f(u)|^2\mu(u)}\\
    &\leq \frac{2\sqrt{\sum_{\{u,v\}\in
          E}w_{uv}|f(u)-s_{uv}f(v)|^2}\, \sqrt{\sum_{\{u,v\}\in
          E}w_{uv}(|f(u)|+|f(v)|)^2}}{\sum_{u\in V}|f(u)|^2\mu(u)}.
  \end{align*}
  Since
  \begin{align*}
    \sum_{\{u,v\}\in E}w_{uv}\left(|f(u)|+|f(v)|\right)^2&\leq 2\sum_{\{u,v\}\in E}w_{uv}(|f(u)|^2+|f(v)|^2)\\
    &=2\sum_{u\in V}\sum_{v,v\sim u}w_{uv}|f(u)|^2,
  \end{align*}
  we conclude that
  \begin{equation}\label{eq:upperI}
    I\leq 2\sqrt{2d_{\mu}\mathcal{R}_{\mu}^s(f)}.
  \end{equation}

  Combining the estimates (\ref{eq:lowerI}) and (\ref{eq:upperI})
  proves the lemma.
\end{proof}

\begin{proof}[Proof of Theorem \ref{thm:CheegerInequality}]
  The upper estimate in (\ref{eq:Cheeger inequality}) follows from
  Lemma \ref{lemma:clusteringI} by setting $f$ to be the eigenfunction
  corresponding to the eigenvalue $\lambda_1(\Delta_{\mu}^s)$.

  It remains to prove the lower estimate of $h_1^s(\mu)$ in
  (\ref{eq:Cheeger inequality}). Let $\widetilde V$ be the subset of~®$V$ that
  achieves the Cheeger constant $h_1^s(\mu)$ in (\ref{eq:Cheeger constant})
  with induced subgraph $(\widetilde V, \widetilde E)$ and
  $\widetilde \tau: \widetilde V\to S_k^1$ be the switching function that
  achieves the frustration index $\iota^s(\widetilde V)$
  in (\ref{eq:frustration index}). Define the function $\widetilde f: V\to \mathbb{C}$
  via:
  \begin{equation}
	\widetilde f(u):=
   		\begin{cases}
			\widetilde \tau(u), & \text{if $u\in \widetilde V$,} \\
        	0, 					& \text{otherwise.}
		\end{cases}
  \end{equation}
  Using (\ref{eq:minmax}) and the estimate
  $|\widetilde \tau(u)-s_{uv}\widetilde \tau(v)|\leq 2$, we obtain
  \begin{align}
    \lambda_1(\Delta_{\mu}^s)&\leq \mathcal{R}_{\mu}^s({\widetilde f})\notag\\
	& =   \frac{\sum_{\{u,v\}\in \widetilde E}w_{uv}|\widetilde \tau(u)-s_{uv}\widetilde \tau(v)|^2+|E(\widetilde V,  \widetilde V^c)|}
			   {\mathrm{vol}_{\mu}(\widetilde V)}\notag\\
    &\leq \frac{2\iota^s(\widetilde V)+|E(\widetilde V,\widetilde V^c)|}
			   {\mathrm{vol}_{\mu}(\widetilde V)}\notag\\
	& \leq 2h_1^s(\mu).\label{eq:lowerbound proof}
  \end{align}
\end{proof}

\begin{Rmk}
  Since the signature is $S_k^1$-valued, the constant $2$ in
  (\ref{eq:lowerbound proof}) can be slightly improved to be
  $|1-\xi^{(k-1)/2}|$ when $k$ is odd.
\end{Rmk}

For $\Gamma=U(1)$ we have the following Cheeger's inequality.

\begin{Thm}\label{thm:CheegerInequalityU(1)}
  Let $G$ be a finite graph with signature $s: E^{or} \to U(1)$. Then
  \begin{equation}\label{eq:Cheeger inequalityU(1)}
    \frac{1}{2}\lambda_1(\Delta_{\mu}^s)\leq h_1^s(\mu)\leq \frac{3}{2}\sqrt{2d_{\mu}\lambda_1(\Delta_{\mu}^s)}.
  \end{equation}
\end{Thm}

The constant in the upper bound of (\ref{eq:Cheeger inequalityU(1)}) is
slightly better than the constant in (\ref{eq:Cheeger
  inequality}). This is due to Lemma \ref{lemma:keyU(1)} below.

For any $t\in (0,1]$, we define $X_{t}: \overline{B_1(0)}\to
\mathbb{C}$ as
\begin{equation}
  X_t(z):=\left\{
    \begin{array}{ll}
      z/|z|, & \hbox{if $z\in \overline{B_1(0)}\setminus B_{t}(0)$,} \\
      0, & \hbox{if $z\in B_{t}(0)$.}
    \end{array}
  \right.
\end{equation}

\begin{Lem}\label{lemma:keyU(1)}
  For any two points $z_1,z_2\in \overline{B_1(0)}$, we have
  \begin{equation}\label{eq:keyU(1)}
    \int_0^1\left|X_{\sqrt{t}}(z_1)-X_{\sqrt{t}}(z_2)\right|\, dt\leq \frac{3}{2}|z_1-z_2|(|z_1|+|z_2|).
  \end{equation}
\end{Lem}

\begin{proof}
  W.l.o.g., we assume that $|z_1|\ge|z_2|>0$. Observe that
  \begin{equation}
    \int_0^1\left|X_{\sqrt{t}}(z_1)-X_{\sqrt{t}}(z_2)\right|\, dt\leq\left|\frac{z_1}{|z_1|}-\frac{z_2}{|z_2|}\right||z_2|^2+(|z_1|^2-|z_2|^2).
  \end{equation}
  Recalling (\ref{eq:keyest1}), we have
  \begin{equation}
    \left|\frac{z_1}{|z_1|}-\frac{z_2}{|z_2|}\right||z_2|^2\leq |z_1-z_2||z_2|\leq\frac{1}{2}|z_1-z_2|(|z_1|+|z_2|).
  \end{equation}
  Combining this with (\ref{eq:keyest2}) proves the lemma.
\end{proof}

With this lemma at hand, the proofs of Theorem
\ref{thm:CheegerInequalityU(1)} and Theorem
\ref{thm:CheegerInequality} are very similar. We omit the details but
mention the following analogue of Lemma \ref{lemma:clusteringI}.

\begin{Lem}\label{lemma:clusteringIU(1)}
  Let $s: E^{or}\to U(1)$ be a signature of $G$ and $f:V\to \mathbb{C}$ be a
  nonzero function. Then there exists $t'\in [0,
  \max_{u\in V}|f(u)|^2]$ such that
  \begin{equation}
    \phi^s_{\mu}(V^f(\sqrt{t'}))\leq \frac{3}{2}\sqrt{2d_{\mu}\mathcal{R}_{\mu}^s(f)}.
  \end{equation}
\end{Lem}
\begin{Rmk}\label{rmk:BSS}
	We notice that the inequality (\ref{eq:Cheeger inequalityU(1)}) for $\Gamma=U(1)$ overlaps with a Cheeger
	inequality for a connection Laplacian of~$G$ discussed by Bandeira, Singer and Spielman \cite{BSS13}
	to solve a partial synchronization problem. The connection Laplacian~$\mathcal{L}$ is defined for a
	simple graph~$G$ where a matrix $O_{uv}\in O(l)$ is assigned to each $(u,v)\in E^{or}$ such that
	$O_{vu}=(O_{uv})^{-1}$. 
	For any vector-valued function $f: V\to \mathbb{R}^l$ and any vertex $u\in V$, we then have
	\begin{equation}
		\mathcal{L}f(u):=\frac{1}{d_u}\sum_{v,v\sim u}w_{uv}(f(u)-O_{uv}f(v))\in \mathbb{R}^l.
	\end{equation}
	For a graph~$G$ with signature $s: E^{or} \to U(1)$ we consider the particular positive measure~$\mu$ on~$V$
	defined as $\mu(u):=d_u$ and rewrite the value $s_{uv}:=a_{uv}+ib_{uv}\in U(1)$ for each $(u,v)\in E^{or}$ as
	\begin{equation}
		\begin{pmatrix}
			a_{uv} & -b_{uv} \\
	    	b_{uv} & a_{uv} \\
		\end{pmatrix}
		\in SO(2).
	\end{equation}
	If we also rewrite a complex valued function $f:=f_1+if_2$ as an $\R^2$-valued function $f:=(f_1, f_2)^T$,
	the discrete magnetic Laplacian~$\Delta_{\mu}^s$ translates into a connection Laplacian $\mathcal{L}^s$
	with eigenvalues
	\begin{equation}
		0\leq \lambda_1(\Delta_{\mu}^s)=\lambda_1(\Delta_{\mu}^s)\leq\cdots\leq \lambda_N(\Delta_{\mu}^s)=\lambda_N(\Delta_{\mu}^s).
	\end{equation}
	Thus, each eigenvalue $\lambda_i(\Delta_{\mu}^s)$ of~$\Delta_{\mu}^s$ is an eigenvalue of $\mathcal{L}^s$ with
	doubled multiplicity. If we denote the Euclidean norm in $\mathbb{R}^l$ by $\Vert\cdot\Vert$, Bandeira,
	Singer and Spielman define a \emph{(partial) $\ell_1$ frustration constant} as
	\begin{equation}
		\eta^*_{G,1}
			:=	\min_{\tau:V\to \mathbb{S}^{l-1}\cup \{0\}}
				\frac{\sum_{u,v\in V}w_{uv}\Vert \tau(u)-O_{uv}\tau(v)\Vert}{\sum_{u\in V}d_u\Vert\tau(u)\Vert},
	\end{equation}
	and prove that
	\begin{equation}\label{eq:Cheeger_BSS}
		\lambda_1(\mathcal{L})\leq \eta^*_{G,1}\leq \sqrt{10\lambda_1(\mathcal{L})}.
	\end{equation}
 	If we assign elements of $SO(2)$ to edges of~$G$ (instead of $O(2)$), we observe that
	\begin{equation}
		\eta^*_{G,1}=2h_1^s(\mu), \qquad\text{and}\qquad \lambda_1(\mathcal{L}^s)=\lambda_1(\Delta_{\mu}^s).
	\end{equation}
	Hence, inequality (\ref{eq:Cheeger_BSS}) leads to inequality (\ref{eq:Cheeger inequalityU(1)}).
	Finally, Bandeira, Singer and Spielman have a refined analysis for (\ref{eq:keyU(1)}) that improves the
	constant $3/2$ in (\ref{eq:keyU(1)}) and (\ref{eq:Cheeger inequalityU(1)}) to $\sqrt{5}/2$, \cite[Appendix A]{BSS13}.
\end{Rmk}

A direct corollary of Theorems \ref{thm:CheegerInequality} and
\ref{thm:CheegerInequalityU(1)} as well as Proposition \ref{pro:Cheeger
  constant} is the following characterization of the case that the
first eigenvalue vanishes.

\begin{Cor}\label{cor:Cheeger constamt}
  $\lambda_1(\Delta_{\mu}^s)=0$ if and only if the underlying graph
  has a balanced connected component.
\end{Cor}

We remark that Corollary \ref{cor:Cheeger constamt} can also be easily derived by the min-max principle (\ref{eq:minmax}).

\section{Spectral clustering via Lens spaces and complex projective
  spaces}\label{section:higher Cheeger}

In this section, we prove the following higher order Cheeger
inequalities.

\begin{Thm}\label{thm:HigherOrder CheegerInequality}
  There exists an absolute constant $C>0$ such that for any finite
  graph $G$ with signature $s$ and all $n \in [N]$,
  we have
  \begin{equation}\label{eq:HigherOrder Cheeger inequality}
    \frac{1}{2}\lambda_n(\Delta_{\mu}^s)\leq h_n^s(\mu)\leq Cn^3\sqrt{d_{\mu}\lambda_n(\Delta_{\mu}^s)}.
  \end{equation}
\end{Thm}

Note that in Theorem \ref{thm:HigherOrder CheegerInequality} the
signature group $\Gamma$ can be either $S_k^1$ or $U(1)$.

The upper bound of $h_n^s(\mu)$ in (\ref{eq:HigherOrder Cheeger
  inequality}) is the essential part of Theorem \ref{thm:HigherOrder
  CheegerInequality} and its proof relies on the development of a
proper spectral clustering algorithm for the operator
$\Delta_{\mu}^s$.  In other words, we aim to find an $n$-subpartition
$\{ V_p \}_{[n]}$ with small constants $\phi_{\mu}^s(V_p)$, based on
the information contained in the eigenfunctions of the operator
$\Delta_{\mu}^s$.

Let $f_i$ be an orthonormal family of eigenfunctions
corresponding to $\lambda_i(\Delta_{\mu}^s)$ for $i \in [n]$. We consider
the following map:
\begin{equation}\label{fct:F}
  F: V\to \mathbb{C}^n, \quad F(u) = (f_1(u), f_2(u), \ldots, f_n(u)).
\end{equation}
Since $\lambda_n(\Delta_{\mu}^s)=\mathcal{R}_{\mu}^s(f_n)$, the
Rayleigh quotient of $F$ is also bounded by $\lambda_n(\Delta_\mu^s)$:
\begin{align}
  \mathcal{R}_{\mu}^s(F)
	:=& \frac{\sum_{\{u,v\}\in E}w_{uv}\Vert F(u)-s_{uv}F(v)\Vert^2}
			{\sum_{u\in V}\mu(u)\Vert F(u)\Vert^2} \notag\\
  	 =& \frac{\sum_{p\in[n]}\sum_{\{u,v\}\in E}w_{uv}|f_p(u)-s_{uv}f_p(v)|^2}
			{\sum_{p\in[n]}\sum_{u\in V}\mu(u)|f_p(u)|^2}\notag\\
   \le& \lambda_n(\Delta_\mu^s),\label{eq:forfinalproof}
\end{align}
where $\Vert\cdot\Vert$ stands for the standard Hermitian norm in
$\mathbb{C}^n$. Our goal is to construct $n$ maps $\Psi_p: V \to {\mathbb
  C}^n$, $p \in [n]$, with pairwise disjoint supports such that
\begin{compactenum}
\item each $\Psi_p$ can be viewed as a localization of $F$, i.e., $\Psi_p$ is the product of $F$ and a cut-off function $\eta: V\to \mathbb{R}$ (see (\ref{eq:localizationF}) below), 
\item each Rayleigh quotient satisfies
  $\mathcal{R}_{\mu}^s(\Psi_p)\leq C(n) \mathcal{R}_{\mu}^s(F)$, where
  $C(n)$ is a constant only depending on $n$.
\end{compactenum}
Then, applying Lemmas \ref{lemma:clusteringI} and \ref{lemma:clusteringIU(1)}
will finish the proof.

This strategy is adapted from the proof of the higher order Cheeger
inequalities for unsigned graphs due to Lee, Oveis Gharan, and
Trevisan \cite{LOT2013}. A critical new point here is to find a proper
metric on the space of points $\{F(u)|u\in V\}\subset \mathbb{C}^n$
for the spectral clustering algorithm. In other words, we need a
proper metric to localize the map $F$. The original algorithm in
\cite{LOT2013} used a spherical metric. The second author \cite{Liu13}
studied a spectral clustering via metrics on real projective spaces to
prove higher order dual Cheeger inequalities for unsigned
graphs. Later in \cite{AtayLiu14}, the above two algorithms and,
hence, the corresponding two kinds of inequalities, were unified in
the framework of Harary's signed graphs, i.e., graphs with signatures
$s: E^{or}\to \{+1,-1\}$. In particular, the metrics on real
projective spaces were shown to be the proper metrics for clustering
in the framework of signed graphs. In our current more general setting
of graphs with signatures $s:E^{or}\to \Gamma$, where $\Gamma=S_k^1$
or $\Gamma=U(1)$, the new metrics will be defined on lens spaces and
complex projective spaces.

\subsection{Lens spaces and complex projective spaces}
In this subsection, we provide metrics of lens spaces and complex
projective spaces for the spectral clustering algorithms in the case
of $\Gamma=S_k^1$ and $\Gamma=U(1)$, respectively. Both lens spaces
and complex projective spaces are important objects in geometry and
topology. See, e.g., \cite[Chapter 5]{Jost05} for details about these
spaces.

Let $\mathbb{S}^{2n-1}:=\{\mathbf{z}\in \mathbb{C}^n\mid
\Vert\mathbf{z}\Vert=1\}$ be the unit sphere in the space
$\mathbb{C}^n$. Then $\Gamma \subset {\mathbb C}$ acts on
$\mathbb{S}^{2n-1}$ by scalar multiplication.
For any two points
$\mathbf{z}_1, \mathbf{z}_2\in \mathbb{S}^{2n-1}\subset \mathbb{C}^n$,
we define the following equivalence relation:
\begin{equation}
  \mathbf{z}_1\sim\mathbf{z}_2\,\,\,\Leftrightarrow\,\,\,\exists\,\gamma\in \Gamma \,\,\text{such that}\,\mathbf{z}_1=\gamma\mathbf{z}_2.
\end{equation}
For $\Gamma = S_k^1$, the corresponding quotient space
$\mathbb{S}^{2n-1}/\Gamma$ is the lens space
$L(k;1,\ldots,1)$, while for $\Gamma=U(1)$, the quotient
space $\mathbb{S}^{2n-1}/\Gamma$ is the complex projective space
$\mathbb{C}P^{n-1}$. Let $[\mathbf{z}]$ denote the equivalence class
of $\mathbf{z}\in \mathbb{S}^{2n-1}$. We consider the following metric
on $\mathbb{S}^{2n-1}/\Gamma$:
\begin{equation}\label{eq:roughmetric}
  d([\mathbf{z}_1], [\mathbf{z}_2]):=\min_{\gamma\in \Gamma}\Vert \mathbf{z}_1-\gamma\mathbf{z}_2\Vert.
\end{equation}

The space $\mathbb{S}^{2n-1}/\Gamma$ can also be endowed with a
distance $d_{quot}$ which is induced from the standard Riemannian
metric on $\mathbb{S}^{2n-1} \subset {\mathbb R}^{2n}$. This induced
metric has positive Ricci curvature. If $\Gamma = S_k^1$, the
sectional curvature of this metric is constant equal to $1$, and if
$\Gamma = U(1)$, this metric is the well-known Fubini-Study
metric. The two metrics $d$ and $d_{quot}$ on
$\mathbb{S}^{2n-1}/\Gamma$ are equivalent, i.e., there exist two
constants $c_1, c_2 > 0$ such that for all $[z_1], [z_2] \in
S^{2n-1}/\Gamma$,
\begin{equation}\label{eq:equivalentMetric}
  c_1d_{quot}([\mathbf{z}_1], [\mathbf{z}_2])\leq d([\mathbf{z}_1], [\mathbf{z}_2])\leq c_2 d_{quot}([\mathbf{z}_1], [\mathbf{z}_2]).
\end{equation}

Recall the
concept of the metric doubling constant $\rho_{\mathbb{X}}$ of a
metric space $(\mathbb{X}, d_{\mathbb{X}})$. This constant is the
infimum of all numbers $\rho$ such that every ball $B$ in $\mathbb{X}$
can be covered by $\rho$ balls of half the radius of $B$.

\begin{Pro}\label{pro:doubling}
  The metric doubling constant $\rho_\Gamma$ of
  $(\mathbb{S}^{2n-1}/\Gamma,d)$ satisfies
  \begin{equation}\label{eq:metricdoubling}
    \log_2\rho_\Gamma\leq Cn,
  \end{equation} where $C$ is an absolute constant.
\end{Pro}

\begin{proof}
  Due to the equivalence (\ref{eq:equivalentMetric}), we only need to
  consider the metric space
  $(\mathbb{S}^{2n-1}/\Gamma,d_{quot})$. Since
  $\mathbb{S}^{2n-1}/\Gamma$ with its standard metric has nonnegative
  Ricci curvature, the Bishop-Gromov comparison theorem guarantees
  \begin{equation}
    \frac{\mathrm{vol}(B_r([\mathbf{z_1}]))}{\mathrm{vol}(B_{r/2}([\mathbf{z_1}]))}\leq \bar{C}^n,
  \end{equation}
  for some absolute constant $\bar{C}$. (Note that the real dimension
  of the lens space is $2n-1$ and of the complex projective space is
  $2n-2$.) A standard argument implies now the claim of the
  proposition. For details see, e.g., \cite[p.67]{CoiWei1971} or
  \cite[Section 2.2]{Liu13}.
\end{proof}

The metric $d$ on $\mathbb{S}^{2n-1}/\Gamma$ induces a pseudo metric
on the space $\mathbb{C}^n\setminus \{0\}$, which -- by abuse of
notation -- will again be denoted by $d$:
\begin{equation}
  d(\mathbf{z}_1, \mathbf{z}_2):=d\left(\left[\frac{\mathbf{z}_1}{\Vert\mathbf{z}_1\Vert}\right],\left[\frac{\mathbf{z}_2}{\Vert\mathbf{z}_2\Vert}\right]\right).
\end{equation}

The following obvious property is the reason why we use the metric $d$
on $S^{2n-1}/\Gamma$ from (\ref{eq:roughmetric}). This reason will
become clear in the next subsection \ref{subsec:localization}.

\begin{Pro}\label{pro:reasonfor metric}
  For every pair $\mathbf{z}_1,\mathbf{z}_2\in \mathbb{C}^n\setminus \{0\}$
  and every $\gamma\in \Gamma$, we have
  \begin{equation}
    d(\mathbf{z}_1,\mathbf{z}_2)=d(\mathbf{z}_1,\gamma\mathbf{z}_2).
  \end{equation}
\end{Pro}

The considerations of the next two subsections prepare the ground for the study of the Rayleigh quotient $\mathcal{R}_{\mu}^s(F)$ of the map
$F: V\to \mathbb{C}^n$ defined in~(\ref{fct:F}).

\subsection{Localization of the map $F$
  } \label{subsec:localization}

We endow the support
$V_F:=\{u\in V|F(u)\neq 0\}$ with the pseudo metric $d_F$ induced by $d$ via
\begin{equation}\label{eq:pseudometric}
 d_F(u,v):=d(F(u), F(v)).
\end{equation}
Given a subset $S\subseteq V$ and $\epsilon>0$, we first define a cut-off function
$\eta: V\to \mathbb{R}$ by
\begin{equation}\label{eq:cutoff}
  \eta(u):=
    \begin{cases}
      0, 												& \text{if $F(u)=0$,} \\
      \max\{0, 1-\frac{1}{\epsilon}d_F(u, S\cap V_F)\}, & \text{otherwise}
    \end{cases}
\end{equation}
and then localize~$F$ via $\eta$ as
\begin{equation}\label{eq:localizationF}
  \Psi:=\eta F: V\to \mathbb{C}^n.
\end{equation}
Note that the $\epsilon$-neighborhood $N_{\epsilon}(S\cap V_F,
d_F):=\{u\in V|d_F(u, S\cap V_F)<\epsilon\}$ of $S\cap V_F$ contains
the support of the map $\Psi$.

In the next lemma, $G_F=(V_F,E_F)$ denotes the induced subgraph on $V_F$ of~$G$.
\begin{Lem}\label{lemma:keyForLocal}
  If $\{u,v\}\in E_F$ and $\Vert F(v)\Vert\leq \Vert F(u)\Vert$ then
  \begin{equation}\label{eq:keyForLocal}
    d(F(u), F(v))\Vert F(v)\Vert\leq \Vert F(u)-s_{uv}F(v)\Vert.
  \end{equation}
\end{Lem}

\begin{proof}
  Observe that we only need to prove
  \begin{equation}\label{eq:keyForLocalProof}
    d(F(u), F(v))\Vert F(v)\Vert\leq \Vert F(u)-F(v)\Vert
  \end{equation}
  for any pair of points $F(u), F(v)\in \mathbb{C}^n\setminus\{0\}$
  with $\Vert F(v)\Vert\leq \Vert F(u)\Vert$: we can replace
  $F(v)$ in (\ref{eq:keyForLocalProof}) by $s_{uv}F(v)$ and use
  Proposition \ref{pro:reasonfor metric} to obtain
  (\ref{eq:keyForLocal}). By the definition of the metric $d$, we obtain
  (\ref{eq:keyForLocalProof}) as follows:
  \begin{align*}
    d(F(u), F(v))\Vert F(v)\Vert
		\leq& \left\Vert\frac{F(u)}{\Vert F(u)\Vert}-\frac{F(v)}{\Vert F(v)\Vert}\right\Vert\Vert F(v)\Vert\\
		\leq& \Vert F(u)-F(v)\Vert,
  \end{align*}
  where we used the estimate~(\ref{eq:keyest1}) for the latter inequality.
\end{proof}

Lemma \ref{lemma:keyForLocal} enables us to prove the following
result.

\begin{Lem}\label{lemma:localization}
  For any $\{u,v\}\in E$, we have
  \begin{equation}\label{eq:localization}
    \Vert\Psi(u)-s_{uv}\Psi(v)\Vert\leq\left(1+\frac{1}{\epsilon}\right) \Vert F(u)-s_{uv}F(v)\Vert.
  \end{equation}
\end{Lem}

\begin{proof}
  If at least one of $F(u)$ and $F(v)$ is equal to zero, then the
  estimate (\ref{eq:localization}) holds trivially. Hence, we suppose
  that $u,v\in V_F$. W.l.o.g., we can assume that $\Vert F(u)\Vert\leq
  \Vert F(v)\Vert$ and calculate
  \begin{align*}
    \Vert\Psi(u)-s_{uv}\Psi(v)\Vert
	   =& \ \Vert\eta(u)F(u)-s_{uv}\eta(v)F(v)\Vert\\
    \leq& \ |\eta(u)|\cdot\Vert F(u)-s_{uv}F(v)\Vert+|\eta(u)-\eta(v)|\cdot\Vert F(v)\Vert\\
    \leq& \ \Vert F(u)-s_{uv}F(v)\Vert+\frac{d_F(u,v)\Vert F(v)\Vert}{\epsilon}.
  \end{align*}
  Applying Lemma \ref{lemma:keyForLocal} completes the proof.
\end{proof}
Note that the inequality (\ref{eq:localization}) is useful for the estimate of the numerator of the Rayleigh quotient of $\Psi$.

\subsection{Decomposition of the underlying space via orthonormal functions}

For later purposes, we work on a general measure space $(\mathcal{V},\mu)$ in this subsection, where $\mathcal{V}$ is a topological space and $\mu$ is a Borel measure. Two particular cases we have in mind are a vertex set $V$ of a finite graph with a measure $\mu: V\to \mathbb{R}^+$, and a closed Riemannian manifold with its Riemannian volume measure. We will apply the results in this subsection to the latter case in Section \ref{section:manifold case}.

On $(\mathcal{V},\mu)$, we further assume that there exist $n$ measurable functions
\begin{equation*}
f_1, f_2, \ldots, f_n: \mathcal{V}\to \mathbb{C},
\end{equation*}
which are orthonormal, i.e., for any $i,j\in [n]$,
\begin{equation*}
\langle f_i, f_j\rangle:=\int_{\mathcal{V}}f_i\overline{f_j}d\mu=\delta_{ij}.
\end{equation*}
Then the map $F: \mathcal{V}\to \mathbb{C}^n$ is given accordingly as in (\ref{fct:F}).

We consider the measure $\mu_F$ on $\mathcal{V}$ given by
\begin{equation*}
d\mu_F=\Vert F\Vert^2d\mu.
\end{equation*}
For any two points $x,y$ in $\mathcal{V}_F:=\{x\in \mathcal{V}: F(x)\neq 0\}$, we have the distance between them
\begin{equation}\label{eq:metricsingeneralspace}
d_F(x,y):=\min_{\gamma\in \Gamma}\left\Vert\frac{F(x)}{\Vert F(x)\Vert}-\gamma\frac{F(y)}{\Vert F(y)\Vert}\right\Vert.
\end{equation}
The main result of this subsection is the following theorem.
\begin{Thm}\label{thm:decomposition}
Let $(\mathcal{V}_F, d_F, \mu_F)$ be as above. There exist an absolute constant $C_0$ and a nontrivial $n$-subpartition $\{T_i\}_{[n]}$ of $\mathcal{V}_F$ such that
 \begin{itemize}
  \item [(i)] $d_F(T_p, T_q)\geq \frac{2}{C_0n^{5/2}}$, for all $p,q\in [n]$, $p\neq q$,
  \item [(ii)] $\mu_F(T_p)\geq \frac{1}{2n}\mu_F(\mathcal{V}_F)$, for all $p\in [n]$.
  \end{itemize}
\end{Thm}
The difficulty for the construction of the above $n$-subpartition is to achieve the property $(ii)$. That is, we have to find a subpartition which possesses large enough measure. When $d_F(x,y)$ is given by the spherical distance $\left\Vert\frac{F(x)}{\Vert F(x)\Vert}-\frac{F(y)}{\Vert F(y)\Vert}\right\Vert$, Theorem \ref{thm:decomposition} was proved in \cite[Lemma 3.5]{LOT2013}. In our situation, we have to deal with the metrics, given in (\ref{eq:metricsingeneralspace}), of lens spaces or complex projective spaces. We refer the reader to \cite{GrigoryanNetrusovYau} for another interesting decomposition result.

An important ingredient of the proof is the following lemma derived from the random partition theory \cite{GKL2003,LeeNaor2005}.
Note that a partition of a set $A$ can also be considered as a map $P:A\to 2^A$, where $x\in A$
is mapped to the unique set $P(x)$ of the partition that contains
$x$. A random partition $\mathcal{P}$ of $A$ is a probability measure $\nu$
on a set of partitions of $A$. Then $\mathcal{P}(x)$ is understood as a random variable from the probability space to subsets of $A$ containing $x$.


\begin{Lem}\label{lemma:paddedpartition}
  Let A be a subset of the metric space $(\mathbb{S}^{2n-1}/\Gamma,
  d)$ (for $d$ recall (\ref{eq:roughmetric})). Then for every $r>0$ and $\delta\in (0,1)$, there
  exists a random partition $\mathcal{P}$ of $A$, i.e., a distribution
  $\nu$ over partitions of $A$ such that
  \begin{itemize}
  \item [(i)] $\mathrm{diam}(S)\leq r$ for any $S$ in every partition
    $P$ in the support of $\nu$,
  \item [(ii)] $\mathbb{P}_{\nu}\left[B_{r/\alpha}(x)\subseteq
      \mathcal{P}(x)\right]\geq 1-\delta$ for all $x\in A$, where
    $\alpha=32\log_2(\rho_{\Gamma})/\delta$.
  \end{itemize}
\end{Lem}
We refer to \cite[Theorem 3.2]{GKL2003} and \cite[Lemma 3.11]{LeeNaor2005} for the proof, see also \cite[Theorem 2.4]{Liu13}.
For convenience, we describe briefly the construction of the random partition claimed in Lemma \ref{lemma:paddedpartition}. Let $\{x_i\}_{[m]}$ be a $r/4$-net of $\mathbb{S}^{2n-1}/\Gamma$, that is, $d(x_i, x_j)\geq r/4$, for any $i\neq j$, and $\mathbb{S}^{2n-1}/\Gamma=\bigcup_{i\in[m]}B_{r/4}(x_i)$. Since $(\mathbb{S}^{2n-1}/\Gamma, d)$ is compact, $m$ is a finite number. For $R\in [r/4,r/2]$, we construct a partition of $(\mathbb{S}^{2n-1}/\Gamma, d)$ as follows. A permutation $\sigma$ of the set $[m]$ provides an order for all points in the net which is used to define, for every $i\in [m]$,
\begin{equation*}
S_{i}^{R,\sigma}:=\{x\in \mathbb{S}^{2n-1}/\Gamma \mid x\in B_R(x_i)\text{ and }\sigma(i)<\sigma(j)\text{ for all } j\in [m] \text{ with } x\in B_R(x_j)\}.
\end{equation*}
That is, we have $x\in S_{i}^{R,\sigma}$ if $\sigma(i)$ is the smallest number for which $x$ is contained in $B_R(x_i)$.
Then $P^{R,\sigma}=\{S_{i}^{R,\sigma}\}_{[m]}$ constitutes a partition of $\mathbb{S}^{2n-1}/\Gamma$. Now let $\sigma$ be a uniformly random permutation of $[m]$, and $R$ be chosen uniformly random from the interval $[r/4,r/2]$. These choices define a random partition $\mathcal{P}$. If we choose $R$ uniformly from a fine enough discretization of the interval $[r/4,r/2]$, we can make $\mathcal{P}$ to be finitely supported. In fact, this random partition fulfills the two properties in Lemma \ref{lemma:paddedpartition}.


\begin{Rmk}
Lemma \ref{lemma:paddedpartition} holds true for any metric space. In particular, the finiteness of the $r/4$-net is not necessary. This is shown in \cite[Lemma 3.11]{LeeNaor2005}.
\end{Rmk}

Lemma \ref{lemma:paddedpartition} leads to the following result. Note that, the property $(ii)$ in Lemma \ref{lemma:paddedpartition} ensures the existence of at least one subpartition which captures a large fraction of the whole measure.
\begin{Lem}\label{lemma:padded_measureversion}
On $(\mathcal{V}_F, d_F, \mu_F)$, for any $r>0$ and $\delta\in (0,1)$, there exists a nontrivial subpartition $\{\widehat{S}_i\}_{[m]}$ such that
 \begin{itemize}
  \item [(i)] $\mathrm{diam}(\widehat{S}_i, d_F)\leq r$ for any $i\in [m]$,
  \item [(ii)] $d_F(\widehat{S}_i, \widehat{S}_j)\geq 2r/\alpha$, where $\alpha=32\log_2(\rho_\Gamma)/\delta$,
  \item [(iii)] $\sum_{i\in [m]}\mu_F(\widehat{S}_i)\geq (1-\delta)\mu_F(\mathcal{V}_F)$.
  \end{itemize}
\end{Lem}
\begin{proof}
Let $\mathcal{P}$ be the random partition on $\mathcal{V}_F$ induced from the one constructed in Lemma \ref{lemma:paddedpartition} via the map $F$. Let $I_{B_{r/\alpha}(x)\subseteq \mathcal{P}(x)}$ be the indicator function for the event that $B_{r/\alpha}(x)\subseteq \mathcal{P}(x)$ happens. Then we obtain from Lemma \ref{lemma:paddedpartition} $(ii)$
\begin{equation}
\mathbb{E}_{\mathcal{P}}\left(\int_{\mathcal{V}}I_{B_{r/\alpha}(x)\subseteq \mathcal{P}(x)}d\mu_F(x)\right)\geq (1-\delta)\mu_F(\mathcal{V})
\end{equation}
by interchanging the expectation and the integral.
On the other hand, we have
\begin{align}
&\mathbb{E}_{\mathcal{P}}\left(\int_{\mathcal{V}}I_{B_{r/\alpha}(x)\subseteq \mathcal{P}(x)}d\mu_F(x)\right)\notag\\
=&\sum_{P\in \mathcal{P}}\sum_{S\in P}\int_{S}I_{B_{r/\alpha}(x)\subseteq \mathcal{P}(x)}d\mu_F(x)\mathbb{P}_{\nu}(P)\notag\\
=&\sum_{P\in \mathcal{P}}\sum_{S\in P}\int_{\widehat{S}}d\mu_F(x)\mathbb{P}_\nu(P),
\end{align}
where $\widehat{S}:=\{x\in S: B_{r/\alpha}(x)\subseteq S\}$. Hence, there exists a partition $P=\{S_i\}_{[m]}$ of $\mathcal{V}_F$ for some natural number $m$ such that
\begin{equation}
\sum_{i\in [m]}\mu_F(\widehat{S}_i)\geq (1-\delta)\mu_F(\mathcal{V}).
\end{equation} This completes the proof.
\end{proof}


In order to prove Theorem \ref{thm:decomposition}, we also need the following result.

\begin{Lem}\label{lemma:spreading}
  If a subset $S\subseteq \mathcal{V}$ satisfies $\mathrm{diam}(S\cap \mathcal{V}_F, d_F)\leq r$
  for some $r\in (0,1)$, then
  \begin{equation}
    \mu_F(S)\leq \frac{1}{n(1-r^2)}\mu_F(\mathcal{V}).
  \end{equation}
\end{Lem}

\begin{proof}
  W.l.o.g., we can assume that $S\subseteq \mathcal{V}_F$. Using the fact
  that $f_1, \ldots, f_n$ are orthonormal, we obtain the following two
  properties. First, we have
  \begin{equation}\label{eq:spreading1}
    \mu_F(\mathcal{V})=\int_{\mathcal{V}}\sum_{p\in [n]}|f_p|^2d\mu=n.
  \end{equation}
  Second, we have for any $\mathbf{z}:=(z_1,z_2,\ldots,z_n)\in
  \mathbb{C}^n$ with $\Vert\mathbf{z}\Vert=1$,
  \begin{equation}\label{eq:spreading2}
    \int_{\mathcal{V}}\left|\langle \mathbf{z}, F(x)\rangle\right|^2d\mu(x)
		= \int_{\mathcal{V}}\sum_{p,q\in [n]} z_p\overline{z_q}\overline{f_p(x)}f_q(x)d\mu(x)=1.
  \end{equation}
  Combining (\ref{eq:spreading1}) and (\ref{eq:spreading2}), we conclude
  for any $y\in S$,
  \begin{align}
    \frac{\mu_F(\mathcal{V})}{n}
		=& \int_{\mathcal{V}}\left|\left\langle
									\frac{F(y)}{\Vert F(y)\Vert}, F(x)
							 \right\rangle\right|^2d\mu(x)\notag\\
		=& \int_{\mathcal{V}} \left|\left\langle
													\frac{F(y)}{\Vert F(y)\Vert}, \frac{F(x)}{\Vert F(x)\Vert}
												\right\rangle\right|^2d\mu_F(x).
  \end{align}
  Since $|z|^2\geq \left(z+\overline{z}\right)^2/4$ for each $z\in \mathbb{C}$, we obtain that for any $\gamma\in
  \Gamma$:
  \begin{align}
    \left|\left\langle \frac{F(y)}{\Vert F(y)\Vert}, \frac{F(x)}{\Vert F(x)\Vert}\right\rangle\right|^2
		=& \ \left|\left\langle
				\frac{F(y)}{\Vert F(y)\Vert}, \gamma\frac{F(x)}{\Vert F(x)\Vert}
			 \right\rangle\right|^2\notag\\
     \geq& \ \frac{1}{4}\left(2-\left\Vert\frac{F(y)}{\Vert F(y)\Vert}-\gamma\frac{F(x)}{\Vert F(x)\Vert}\right\Vert^2\right)^2.
  \end{align}
  Recalling (\ref{eq:metricsingeneralspace}), the definition of $d_F$, we arrive at
  \begin{equation}
    \frac{\mu_F(\mathcal{V})}{n}\geq\int_{S}\left(1-\frac{1}{2}d_F(y,x)^2\right)^2d\mu_F(x)\geq (1-r^2)\mu_F(S).
  \end{equation}
\end{proof}

\begin{proof}[Proof of Theorem \ref{thm:decomposition}]
With Lemma \ref{lemma:padded_measureversion} and Lemma \ref{lemma:spreading} at hand, Theorem \ref{thm:decomposition} can be proved similarly as
\cite[Lemma 3.5]{LOT2013}, see also \cite[Lemma 6.2]{Liu13}. For convenience, we recall it here. Let $\{\widehat{S}_i\}_{[m]}$ be the subpartition constructed in Lemma \ref{lemma:padded_measureversion}. Then by Lemma \ref{lemma:spreading}, we have for each $i\in [m]$,
\begin{equation}
\mu_F(\widehat{S}_i)\leq\frac{1}{n(1-r^2)}\mu_F(\mathcal{V}).
\end{equation}
We apply the following procedure to $\{\widehat{S}_i\}_{[m]}$. If we can find two of them, say $\widehat{S}_i$ and $\widehat{S}_j$, such that
\begin{equation*}
\mu_F(\widehat{S}_i)\leq \frac{1}{2n}\mu_F(\mathcal{V}),\,\,\,\, \mu_F(\widehat{S}_j)\leq \frac{1}{2n}\mu_F(\mathcal{V}),
\end{equation*}
then replace them by $\widehat{S}_i\cup\widehat{S}_j$. Thus, when we stop, we obtain the sets $T_1, T_2, \ldots, T_l$ for some number $l$, such that
\begin{equation*}
\mu_F(T_i)\leq\frac{1}{n(1-r^2)}\mu_F(\mathcal{V}), \,\,\,\forall\, i\in [l],
\end{equation*}
and
\begin{equation*}
\mu_F(T_i)\geq\frac{1}{2n}\mu_F(\mathcal{V}), \,\,\,\forall\, i\in [l-1].
\end{equation*}
Setting $r=\frac{1}{3\sqrt{n}}$ and $\delta=\frac{1}{4n}$, we check that
\begin{equation}
(n-1)\cdot\frac{1}{n(1-r^2)}<1-\delta-\frac{1}{2n}.
\end{equation}
This implies that $l\geq n$. Moreover, if we redefine $T_n:=\bigcup_{j=n}^lT_j$, we have
\begin{equation}
\mu_F(T_n)\geq\frac{1}{2n}\mu_F(\mathcal{V}).
\end{equation}
Thus the subpartition $\{T_i\}_{[n]}$ satisfies the property $(ii)$. One can then verify the property $(i)$ by Proposition \ref{pro:doubling}  and Lemma \ref{lemma:padded_measureversion}.
\end{proof}

%

\subsection{Proof of Theorem \ref{thm:HigherOrder CheegerInequality}}\label{section:proofof higher}

We first prove the upper bound of (\ref{eq:HigherOrder Cheeger inequality}). Let $\{T_i\}_{[n]}$ be the subpartition of $V_F$ obtained from Theorem \ref{thm:decomposition}. Choosing $\epsilon=\frac{1}{C_0n^{5/2}}$, we define the cut-off
  functions $\eta_{p}$ as in (\ref{eq:cutoff}) (replacing
  the set $S$ there by $T_p$). Then the maps $\Psi_p:=\eta_{p}F$,
  $p\in [n]$, have pairwise disjoint support. Recalling that
  $\Psi_p|_{T_p}=F|_{T_p}$, and applying Lemmas
  \ref{lemma:localization} as well as fact $(ii)$ of Theorem \ref{thm:decomposition}, we obtain that
  for any $p\in [n]$,
  \begin{align}
    \mathcal{R}_{\mu}^s(\Psi_p)&\leq \left(1+\frac{1}{\epsilon}\right)^2\frac{\sum_{\{u,v\}\in E}w_{uv}\Vert F(u)-s_{uv}F(v)\Vert^2}{\sum_{u\in T_p}\mu(u)\Vert F(u)\Vert^2}\notag\\
    &\leq 2n(1+C_0n^{5/2})^2\mathcal{R}_{\mu}^s(F)\leq Cn^6\mathcal{R}_{\mu}^s(F),
  \end{align}
  where $C$ is an absolute constant.
For every $p\in [n]$, the map $\Psi_p$ has
at least one coordinate function $\psi_p$ that satisfies $\mathcal{R}_{\mu}^s(\psi_p)
\leq \mathcal{R}_{\mu}^s(\Psi_p)$. In particular, we find functions $\psi_p$,
$p\in [n]$, with pairwise disjoint support and an absolute constant $C$ such that
\begin{equation}
  \mathcal{R}_{\mu}^s(\psi_p)\leq Cn^6 \mathcal{R}_{\mu}^s(F).
\end{equation}
Now inequality (\ref{eq:forfinalproof}) and Lemma
\ref{lemma:clusteringI} for $\Gamma = S_k^1$ or Lemma
\ref{lemma:clusteringIU(1)} for $\Gamma = U(1)$ yield the desired
upper bound of (\ref{eq:HigherOrder Cheeger inequality}).

Now we prove the lower bound of (\ref{eq:HigherOrder Cheeger inequality}).
Suppose that the $n$-way Cheeger constant $h_n^{s}(\mu)$ is achieved by the
nontrivial $n$-subpartition $\{\widetilde{V}_p\}_{[n]}$ and that the function
$\widetilde{\tau}_p: \widetilde{V}_p\to \Gamma$ achieves the frustration
index $\iota^s(\widetilde V_p)$ for each $p\in[n]$. Moreover, consider
functions $\widetilde{f}_p: V \to {\mathbb C}$ with pairwise disjoint support given
for $p\in [n]$ by:
\begin{equation}
 \widetilde{f}_p(u):=\begin{cases}
                   		\widetilde{\tau}_p(u),	& \text{if $u\in \widetilde{V}_p$;} \\
                   		0, 						& \text{otherwise.}
                 \end{cases}
\end{equation}
By the min-max principle (\ref{eq:minmax}), we know
\begin{equation}\label{eq:Higher lower1}
  \lambda_n(\Delta_{\mu}^s)\leq \max_{a_1,\ldots, a_n}\mathcal{R}_{\mu}^s(\widetilde{f}_a),
\end{equation}
where the maximum is taken over all complex numbers $a_1, \ldots , a_n \in \C$ such that
$\widetilde{f}_a:=\sum_{p\in [n]}a_p\widetilde{f}_p$ is a nontrivial
linear combination of $\widetilde{f}_1,\ldots, \widetilde{f}_n$. This implies
\begin{equation}\label{eq:Higher lower2}
 \sum_{u\in V}\mu(u)|\widetilde{f}_a(u)|^2=\sum_{p\in [n]}|a_p|^2\mathrm{vol}_{\mu}(\widetilde{V}_p).
\end{equation}
We now want to relate \eqref{eq:Higher lower1} and \eqref{eq:Higher lower2} to the frustration
index and the boundary measure. To that direction, we set $B_{uv}:=w_{uv}|\widetilde{f}_a(u)-s_{uv}\widetilde{f}_a(v)|^2$
and obtain
\begin{equation*}
  \sum_{\{u,v\}\in E} B_{uv}
		= \frac{1}{2} \sum_{p,q \in [n]} \sum_{\substack{u \in \widetilde{V}_p\\v \in \widetilde{V}_q}} B_{uv}
		  + \sum_{p\in [n]}\sum_{\substack{u \in \widetilde{V}_p\\v \in V^*}} B_{uv}
		  + \frac{1}{2}\sum_{u,v \in V^*} B_{uv},
\end{equation*}
where $V^* = \left( \bigcup_{p \in [n]} \widetilde{V}_p \right)^c$.
For $u,v \in \widetilde{V}_p$, $p \in [n]$, we have
\begin{equation}
  |\widetilde{f}_a(u)-s_{uv}\widetilde{f}_a(v)|^2=|a_p|^2 \cdot |\widetilde{\tau}_p(u)-s_{uv}\widetilde{\tau}_p(v)|^2,
\end{equation}
while for $u\in \widetilde{V}_p$ and $v\in \widetilde{V}_q$ with $p,q \in [n]$ and
$p\neq q$ we have
\begin{equation}
  |\widetilde{f}_a(u)-s_{uv}\widetilde{f}_a(v)|^2=|a_p\widetilde{\tau}_p(u)-s_{uv}a_q\widetilde{\tau}_q(v)|^2\leq 2(|a_p|^2+|a_q|^2).
\end{equation}
Now the definition of the frustration index and of the boundary measure yield
\begin{align}
  \sum_{\{u,v\}\in E} B_{uv}
  \leq& \, \sum_{p\in [n]}|a_p|^2\,
				\left(2\iota^s(\widetilde{V}_p)
					  \, +\, 2\big|E(\widetilde{V}_p,\bigcup_{q\neq p}\widetilde{V}_q)\big|
					  \, +\, \left|E(\widetilde{V}_p, V^*)\right|\right)\notag\\
  \leq& \  2\, \sum_{p\in [n]}|a_p|^2\,
					\left(\iota^s(\widetilde{V}_p)
						\, +\, \left|E(\widetilde{V}_p,\widetilde{V}_p^c)\right|\right).\label{eq:Higher lower3}
\end{align}
If we now combine the estimates (\ref{eq:Higher lower1}), (\ref{eq:Higher lower2}), and (\ref{eq:Higher lower3}), we arrive at
\begin{equation}
  \lambda_n(\Delta_{\mu}^s)\leq 2\max_{p\in [n]}\phi_{\mu}^s(\widetilde{V}_p)=2h_n^s(\mu).
\end{equation}

\section{Application: Spectral clustering on oriented graphs and
  mixed graphs}\label{section:Direct graphs}

In this section, we discuss an application of the Cheeger inequalities
(and their proofs) in the case $\Gamma = S_k^1$. These results indicate
algorithms to find interesting substructures in an oriented graph or a
mixed graph.

\subsection{Generalization of Harary's balance theorem}

Let us first discuss an equivalent definition of the Cheeger constant
$h_1^s(\mu)$ if $\Gamma=S_k^1$. For a nonempty subset $\widetilde V$ of $V$,
let $\widetilde V_{0},\ldots,\widetilde V_{k-1}$ be an \emph{ordered $k$-partition}
of $\widetilde V$, that is, $\widetilde V_{i}$ are pairwise disjoint sets
and their union is $\widetilde V$. In contrast to a
nontrivial $k$-partition, all but one $\widetilde V_{i}$ may be empty. We
write $\mathscr V_k(\widetilde V)$ for an ordered $k$-partition $\widetilde V_{0},\ldots,\widetilde V_{k-1}$
of~$\widetilde V$.

Given an ordered $k$-partition $\mathscr V_k(\widetilde V)$ of $\widetilde V \subseteq V$, we define,
for $0\leq i,j \leq k-1$ and $l\in \Z$,
\begin{equation}
  |E^l(\widetilde V_i, \widetilde V_j)|
	:=\sum_{u\in \widetilde V_i}\ \sum_{\substack{v\in \widetilde V_j \text{ s.t.} \\ s_{uv}=\xi^l}}w_{uv}
\end{equation}
as the (weighted) cardinality of oriented edges with signature~$\xi^l$ that begin in~$\widetilde V_i$ and
terminate in $\widetilde V_j$.

\begin{Def} Let $G$ be a finite graph with signature $s:E^{or}\to S_k^1$.
	For any nonempty subset $\widetilde V$ of $V$, the \emph{$k$-partiteness ratio}
	of an ordered $k$-partition $\mathscr V_k(\widetilde V)$ of $\widetilde V$ is defined as
  \begin{equation}\label{eq:kpartite ratio}
    \beta^s_{\mu}\left(\mathscr V_k(\widetilde V)\right)
		= \frac{\frac{1}{2}\sum_{i,j=0}^{k-1} \sum_{l=1}^{k-1}
								|1-\xi^l| \cdot |E^{i-j+l}(\widetilde V_i, \widetilde V_j)| + |E(\widetilde V,\widetilde V^c)|}
			   {\mathrm{vol}_{\mu}(\widetilde V)}.
  \end{equation}
  The \emph{minimal $k$-partiteness ratio} $\beta_{\mu}^{s}(\widetilde V,k)$ of $\widetilde V$ is
  defined as
  \begin{equation}\label{eq:minkpartite ratio}
    \beta_{\mu}^{s}(\widetilde V,k)
		:= \min_{\mathscr V_k(\widetilde V)} \beta^s_{\mu} \left(\mathscr V_k(\widetilde V)\right),
  \end{equation}
  where the minimum is taken over all ordered $k$-partitions $\mathscr V_k(\widetilde V)$ of $\widetilde V$.
\end{Def}

The next goal is to prove that the Cheeger constant for $\Gamma = S_k^1$ can also
be expressed in terms of the k-partiteness ratio, see Corollary \ref{h-bipart} below.

\begin{Lem}\label{lemma:Cheeger constantnew}
  Let $G$ be a finite graph with signature $s:E^{or} \to S_k^1$.
  For any nonempty $\widetilde V \subseteq V$, we have
  \begin{equation}
    \phi_{\mu}^s(\widetilde V)=\beta_{\mu}^s(\widetilde V, k).
  \end{equation}
\end{Lem}

\begin{proof}
  For any function $\tau: V_1\to S_k^1$, we have a natural $k$-partition $\mathscr V_k(\widetilde V)$
  of $\widetilde V$ given by
  \begin{equation}\label{eq:kpartition}
    \widetilde V_i := \{ u\in \widetilde V \mid \tau(u)=\xi^i\}
  \end{equation}
  for $i=0,1,\ldots,k-1$. We can check that
  \begin{equation}
    \sum_{\{u,v\}\in \widetilde E}
		w_{uv} |\tau(u)-s_{uv}\tau(v)|
	= \frac{1}{2} \sum_{i,j=0}^{k-1} \sum_{l=1}^{k-1} |1-\xi^l| \cdot |E^{i-j+l}(\widetilde V_i, \widetilde V_j)|.
  \end{equation}
  Observe that the correspondence between the set of $S_k^1$-valued
  functions on $\widetilde V$ and the set of ordered $k$-partitions of $\widetilde V$
  given by (\ref{eq:kpartition}) is one-to-one. Hence, we obtain by definition of the
  frustration index
  \begin{equation}
    \iota^s(\widetilde V)
		= \min_{\mathscr V_k(\widetilde V)}\
			\frac{1}{2} \sum_{i,j=0}^{k-1} \sum_{l=1}^{k-1}
							|1-\xi^l|\cdot|E^{i-j+l}(\widetilde V_i, \widetilde V_j)|.
  \end{equation}
  This proves the lemma.
\end{proof}

\begin{Cor} \label{h-bipart}
  Let $G$ be a finite graph with signature $s: E^{or} \to S_k^1$. Then
  \begin{equation}\label{eq:Cheeger constantnew}
    h_1^s(\mu)=\min_{\emptyset\neq \widetilde V \subseteq V}\beta_{\mu}^s(\widetilde V, k).
  \end{equation}
\end{Cor}

This enables us to prove the following structural balance theorem.

\begin{Thm}\label{thm:balanceThm}
  Let $G$ be a finite connected graph with a signature $s: E^{or} \to
  S_k^1$. Then the following statements are equivalent:
  \begin{itemize}
  \item[(i)] The signature $s$ is balanced.
  \item[(ii)] There exists an ordered $k$-partition $V_0, \ldots, V_{k-1}$ of $V$
	such that all edges that begin in~$V_i$ and terminate in $V_j$ have
	signature $\xi^{i-j}$ for all $0\leq i,j\leq k-1$.
  \end{itemize}
\end{Thm}

\begin{proof}
  Recall that $h_1^s(\mu)=0$ if and only if the signature is
  balanced. The theorem is then a direct consequence of
  (\ref{eq:Cheeger constantnew}).
\end{proof}

\begin{Rmk}
  Harary's balance theorem \cite{Harary53} states that a signature
  $s:E^{or}\to \{\pm1\}$ is balanced if and only if there exists a
  bipartition $V_0, V_1$ of $V$ such that an edge has signature $-1$
  if and only if it has one end point in $V_0$ and one in
  $V_1$. Theorem \ref{thm:balanceThm} is a natural generalization of
  Harary's theorem.
\end{Rmk}

\begin{figure}[b]
\hspace{0.05\linewidth}
\begin{minipage}[t]{0.4\linewidth}
\centering
\includegraphics[width=0.9\textwidth]{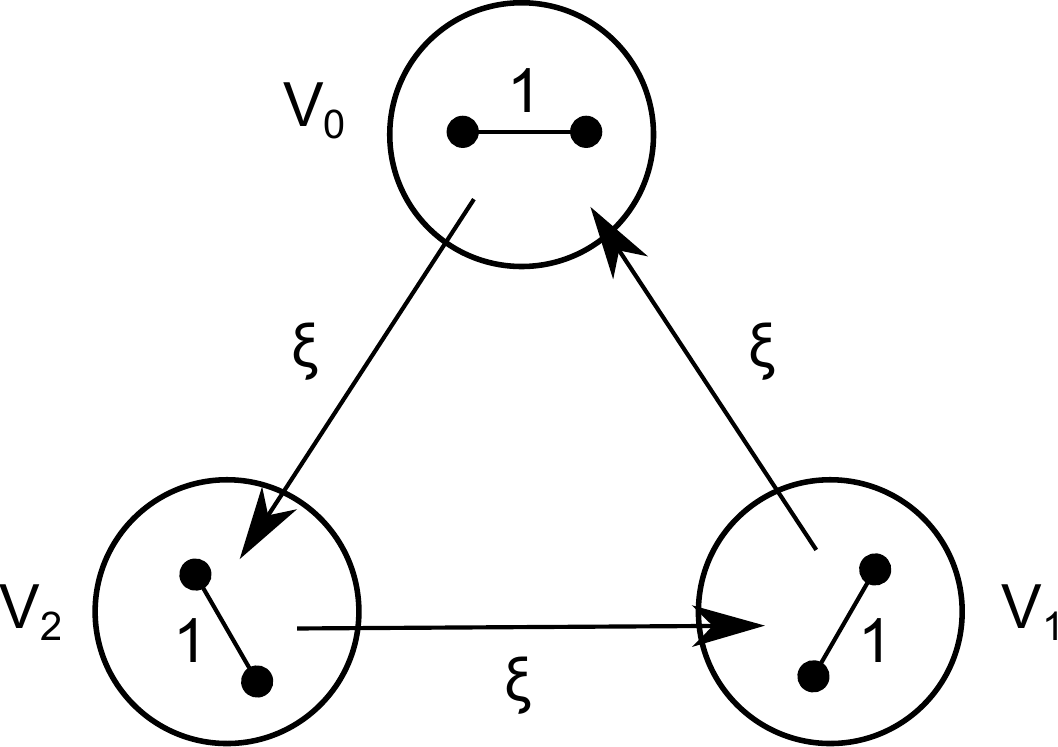}
\end{minipage}
\hfill
\begin{minipage}[t]{0.4\linewidth}
\centering
\includegraphics[width=0.9\textwidth]{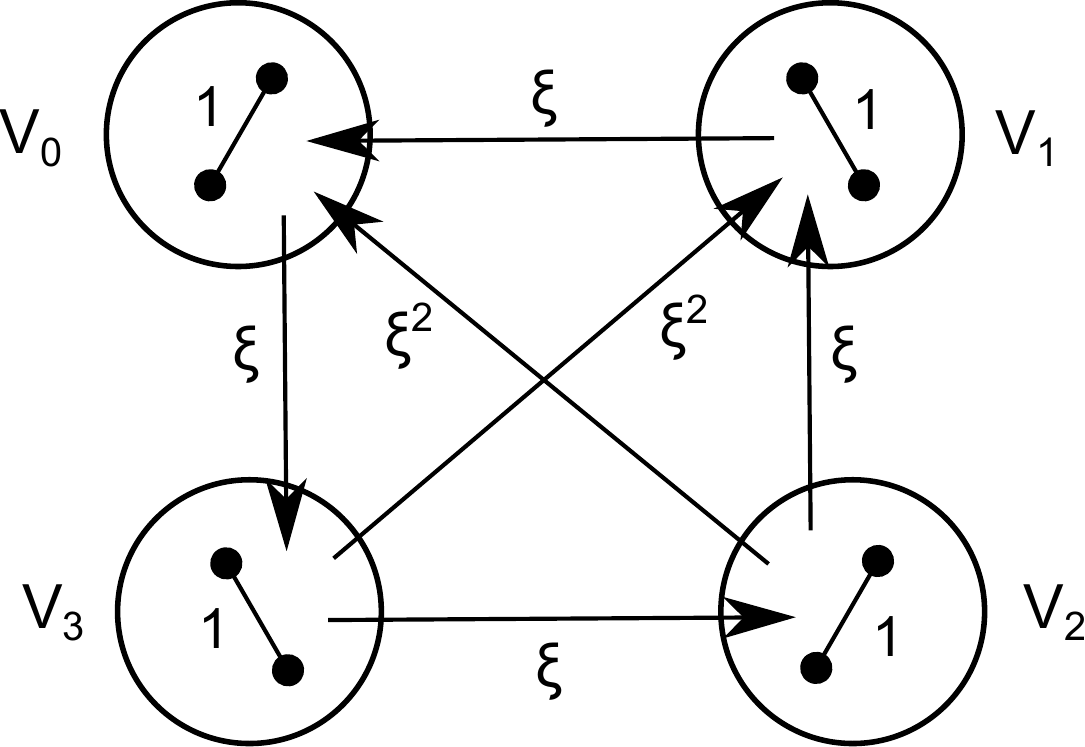}
\end{minipage}
\hspace{0.1\linewidth}
\caption{Schematic illustration of Theorem~\ref{thm:balanceThm} for $k = 3$ (left)
and $k= 4$ (right).}\label{fig:balanced}
\end{figure}

In Figure~\ref{fig:balanced}, we schematically illustrate the situation of
Theorem~\ref{thm:balanceThm} if $k\in\{ 3, 4\}$. The class of edges that
begin and terminate in ~$V_i$ are represented by one unoriented edge labeled
by~$\xi^0=1$. For distinct $i,j$, the class of
edges with endpoints in $V_i$ and $V_j$ are represented by an oriented edge
that begins in $V_i$ and terminates in~$V_j$ with~$i<j$. These oriented edges
 are labeled by $\xi^{i-j}$.

\subsection{Finding a good substructure}\label{subsec:finding_substructure}

The proof of Cheeger's inequality in Section~\ref{section:Cheeger inequality},
especially Lemma \ref{lemma:clusteringI}, actually indicates an algorithm to
find a subset $\widetilde V \subseteq V$ with a constant $\phi_{\mu}^s(\widetilde V)$
close to the Cheeger constant $h_1^s(\mu)$ of~$G$. In other words,  $\phi_{\mu}^s(\widetilde V)$
is not larger than the upper bound for $h_1^s(\mu)$ given in Cheeger's inequality
(Theorem \ref{thm:CheegerInequality}): for every nonzero function $f: V \to {\mathbb C}$,
Lemma \ref{lemma:clusteringI} provides a nonempty subset $\widetilde V := V^f(\sqrt{t'}) \subseteq V$
satisfying \eqref{eq:phisR}. If we choose $f$ to be the eigenfunction corresponding
to $\lambda_1(\Delta_\mu^s)$, we see that $\widetilde V$ is a nonempty subset of~$V$
with the required property.

Now consider a finite graph $G$ with a $k$-cyclic signature $s$. From Lemma~\ref{lemma:Cheeger constantnew},
we know that $\phi_\mu^s(\widetilde V)$ agrees with
the minimum of the $k$-partiteness ratios of all ordered $k$-partitions~$\mathscr V_k(\widetilde V)$.
Having found a nonempty subset $\widetilde V := V^f(\sqrt{t'}) \subseteq V$
satisfying \eqref{eq:phisR}, we explain in this subsection, how to find
a finer substructure of $\widetilde V$, namely an ordered $k$-partition $\mathscr V_k(\widetilde V)$
with a $k$-partiteness ratio that is at most the upper bound given in \eqref{eq:phisR}.
The precise statement is given in Proposition \ref{prop:clusteringImodify} below.

Recall the notation $Q_{j}^{\theta}$ and $V^f(t)$ of (\ref{eq:notationQ}) and
(\ref{eq:notationVft}), respectively. Given $t\in [0,1]$ and $\theta\in [0,2\pi)$,
we define an ordered $k$-partition $\mathscr V_k(V^f(\sqrt{t},\theta))$ of $V^f(\sqrt{t},\theta)\subseteq V$ by
\begin{equation}\label{eq:kpartite algorithm}
  V_j^f(\sqrt{t}, \theta):=\{u\in V \mid \sqrt{t} \leq |f(u)| \text{ and } f(u)\in Q_{j}^{\theta}\}
\end{equation}
for $0\leq j \leq k-1$ and modify Lemma \ref{lemma:clusteringI} into the following result.

\begin{Pro}\label{prop:clusteringImodify}
  Let $s: E^{or}\to S_k^1$ be a signature of~$G$. For any nonzero function
  $f:V\to \mathbb{C}$ with $\max_{u\in V}|f(u)|=1$, there exist
  $t'\in [0, 1]$ and $\theta'\in [0,2\pi)$ such that
  \begin{equation}
    \beta^s_{\mu}\left(\mathscr V_k(V^f(\sqrt{t'},\theta')\right) \leq 2\sqrt{2d_{\mu}\mathcal{R}_{\mu}^s(f)}.
  \end{equation}
\end{Pro}

\begin{proof}
  Instead of inequality (\ref{eq:estimatemodify}), we consider the equality
  \begin{align}
    \frac{1}{2\pi}	\int_0^{2\pi}& \int_0^1
						\sum_{\{u,v\}\in E}
							w_{uv}\left|Y_{\sqrt{t}, \theta}(f(u))-s_{uv}Y_{\sqrt{t}, \theta}(f(v))\right|
					\, dt\, d\theta\notag\\
    = &\frac{1}{2\pi} \int_0^{2\pi} \int_0^1
							\left( \frac{1}{2}
										\sum_{i,j=0}^{k-1}
										\sum_{l=1}^{k-1}
											\left|
											1-\xi^l\right|\cdot\left|E^{i-j+l}(W_i, W_j)
											\right|
									+ \left|
										E( \widetilde V, \widetilde V^c)
									  \right|
							\right)\, dt\, d\theta.\notag
  \end{align}
  where $W_j := V_j^f(\sqrt{t}, \theta)$ and $\widetilde V := V^f(\sqrt{t})$.
%
  The remaining proof follows along similar arguments as the ones
  given in the proof of Lemma \ref{lemma:clusteringI}.
\end{proof}

This Proposition provides the following spectral clustering
algorithm to find an ordered $k$-subpartition of $V$ with a
$k$-partiteness ratio bounded above by the upper bound in Cheeger's
inequality. Firstly, find the eigenfunction $f_1: V\to \mathbb{C}$
corresponding to $\lambda_1(\Delta_{\mu}^s)$. For convenience, we can
normalize $f_1$ such that $\max_{u\in V}|f(u)|=1$. Secondly, find the
required ordered $k$-subpartion from the sets (\ref{eq:kpartite
  algorithm}) by running over fine enough discretizations of the
parameters $t$ and $\theta$.

\subsection{Applications to partially oriented graphs}\label{subsec:applications}

In this subsection, we consider \emph{mixed graphs} instead of undirected graphs
which are studied in scheduling problems, for example~\cite{Sotskov00,Ries07}.
Recall that a \emph{mixed graph} is a graph $G=(V, E_U\cup E_O)$ that consists
of unoriented edges (the set $E_U$) as well as oriented edges (the set $E_O$)
such that no two vertices $u,v \in V$ form more than one edge of $E_U \cup E_O$.
As mentioned in the introduction, we call such a graph also {\em partially oriented}.
Clearly, a partially oriented graph is an \emph{oriented graph} if and only if
$E_U=\emptyset$. The algorithm discussed in the previous subsection has interesting
applications for partially oriented graphs.

Given a partially oriented graph $G=(V, E_U\cup E_O)$ and a natural number~$k$, we
now want to find a nonempty subset $\widetilde V \subseteq  V$ and an ordered $k$-subpartition
$\mathscr{V}_k(\widetilde V)=\{V_0, V_1, \ldots, V_{k-1}\}$ of~$\widetilde V$ which
approximates the following ideal substructure:
\begin{itemize}
\item[(i)] The subset $\widetilde V$ has empty boundary.
\item[(ii)] An edge $e\in E_U \cup E_O$ with endpoints $u,v\in V_i$ for some~$0\leq i\leq k-1$
	is unoriented, that is, $e\in E_U$.
\item[(iii)] The partially oriented subgraph~$G_{\widetilde V}$ induced by~$\widetilde V$
	has the following \emph{cyclic property}: the only oriented edges of $G_{\widetilde V}$
	begin in $V_i$ and end in $V_{i-1}$ for some $0\leq i\leq k-1$ where we identify $V_{-1}$ and $V_{k-1}$.
\end{itemize}
Such ideal substructures are schematically illustrated in Figure~\ref{fig:balance_mixed}
for $k=3$ and $k=4$.

Our approach to this problem is to construct an unoriented graph $G =(V,E)$ with a
$k$-cyclic signature $s$ from a given partially oriented graph $G=(V,E_U \cup E_O)$.
More precisely, we consider the new edge set $E:=E_U\cup E_O$ where the orientations
in $E_O$ are dropped and define a signature $s: E^{or}\to S_k^1$ by assigning to every
edge $\{u,v\}\in E$ the value
\begin{equation}
s_{uv}:=\begin{cases}
           1, 		& \text{if $\{u,v\}\in E_U$;} \\
           \xi, 	& \text{if $(u,v)\in E_O$;}\\
		   \xi^{-1}	& \text{if $(v,u)\in E_O$.}
         \end{cases}
\end{equation}
This construction to transform a connected partially oriented graph $G$ is set up in such a way that the signature
is balanced if and only if $G$ has the above ideal structure.
Using the eigenfunction of the eigenvalue $\lambda_1(\Delta_{\mu}^s)$, we
apply the spectral clustering algorithm discussed in the Section~\ref{subsec:finding_substructure}
to find a $k$-subpartition $\mathscr V_k(\widetilde V)$ of
some $\widetilde V \subseteq V$ with $k$-partiteness ratio $\beta_\mu^s(\mathscr V_k(\widetilde V))$
at most the upper bound given in Cheeger's inequality. Note that the
$k$-partiteness ratio can be viewed as a measure to quantify the quality
of an approximation to the ideal case which is achieved if and only if
$\beta_\mu^s(\mathscr V_k(\widetilde V)) = 0$. By Corollary~\ref{h-bipart},
the $k$-partiteness ratio $\beta_\mu^s(\mathscr V_k(\widetilde V))$ is bounded
from below by the Cheeger constant $h_1^s(\mu)$.
\begin{figure}
\hspace{0.05\linewidth}
\begin{minipage}[t]{0.4\linewidth}
\centering
\includegraphics[width=0.9\textwidth]{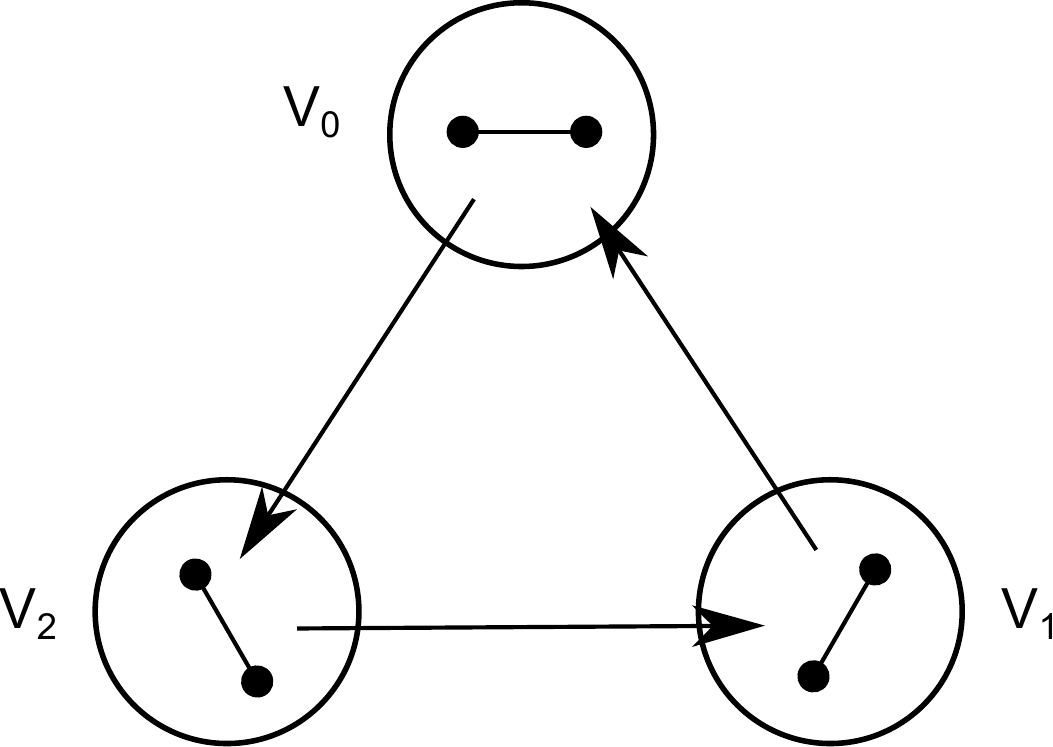}

\end{minipage}
\hfill
\begin{minipage}[t]{0.4\linewidth}
\centering
\includegraphics[width=0.9\textwidth]{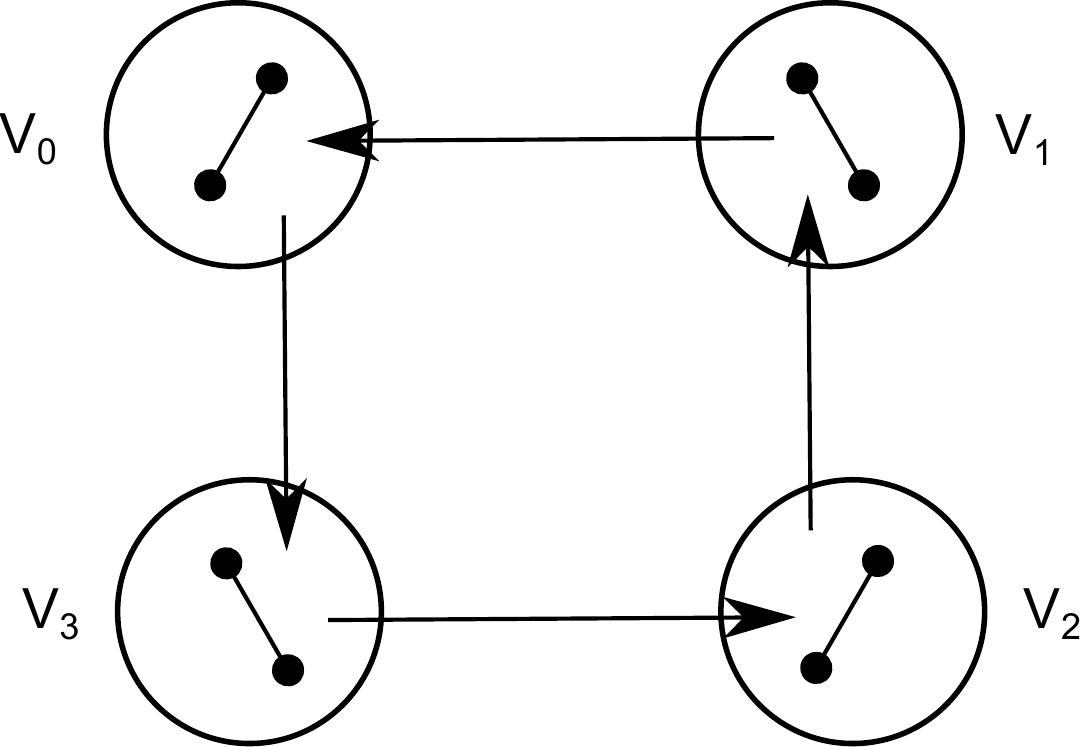}

\end{minipage}
\hspace{0.1\linewidth}
\caption{Ideal substructures of partially oriented graphs in case of $k=3$ (left)
$k=4$ (right) that are approximated in Section~\ref{subsec:applications}.}
\label{fig:balance_mixed}
\end{figure}

We remark that in the special situation were we start with an oriented graph, the ordered $k$-subpartition
$V_0, V_1, \ldots, V_{k-1}$ of~$V$ approximates
an ideal substructure with no edges having both endpoints in $V_i$ for some $0\leq i\leq k-1$.

These considerations can clearly be extended to obtain multi-way
spectral clustering algorithms. Combining the method here with the spectral clustering
via metrics on lens spaces in
Section \ref{section:higher Cheeger}, we can find $n$ subgraphs where each subgraph
defines a sparse cut and approximates an ideal substructure as
described above.


\section{Magnetic Laplacians on Riemannian
  manifolds}\label{section:manifold case}

In this section, we transfer the ideas related to Cheeger constants
and Cheeger inequalities from discrete magnetic Laplacians to the
Riemannian setting.

Let $M$ be a closed connected Riemannian manifold. We consider a real
smooth $1$-form $\boldsymbol{\alpha}$ and the corresponding
\emph{magnetic Laplacian} $\Delta^{\boldsymbol{\alpha}}$ on $M$,
defined as
\begin{equation}\label{eq:quadratic form}
  \Delta^{\boldsymbol{\alpha}}=D^*D,
\end{equation}
where the operator $D:=d+i\boldsymbol{\alpha}$ maps smooth complex
valued functions to smooth complex valued $1$-forms and $D^*$ is the
formal adjoint of $D$ w.r.t. the $L^2$ inner product of functions and
$1$-forms:
\begin{equation}
  \int_M\langle Df, \boldsymbol{\eta}\rangle dx=\int_Mf\overline{D^*\boldsymbol{\eta}}dx.
\end{equation}
The $1$-form $\boldsymbol{\alpha}$ is called the \emph{magnetic
  potential}. One can check that for any smooth function $f: M\to
\mathbb{C}$,
\begin{equation}
  \Delta^{\boldsymbol{\alpha}}f:=\Delta f-2i\langle df, \boldsymbol{\alpha}\rangle+\left( i d^*\boldsymbol{\alpha} + |\boldsymbol{\alpha}|^2\right)f,
\end{equation}
where $d$ is the exterior differential, $d^*$ its formal adjoint,
$\Delta:=d^*d$ is the Laplace-Beltrami operator, $\langle\cdot,
\cdot\rangle$ the Hermitian inner product in the cotangent bundle
$T^*M$ induced by the Riemannian metric, and
$|\boldsymbol{\alpha}|^2:=\langle\boldsymbol{\alpha},\boldsymbol{\alpha}\rangle$.

We recall some basic spectral properties of the magnetic Laplacian
from \cite{Shigekawa87} (see also \cite[Section 4]{Paternain01}).  The
operator $\Delta^{\boldsymbol{\alpha}}$ is essentially self-adjoint as
an operator defined on smooth complex valued functions (with compact
support). Its self-adjoint extension is defined on a dense subset of
the Hilbert space $L^2(M, \mathbb{C})$ of complex valued square
integrable functions w.r.t the Riemannian measure. In the sequel, we
will use the same notation for both the essentially self-adjoint
operator and its closed self-adjoint extension. Since $M$ is compact,
$\Delta^\alpha$ has only discrete spectrum, and the eigenvalues can be
listed with multiplicity as follows (see \cite[Theorem
2.1]{Shigekawa87})
\begin{equation}
  0\leq \lambda_1(\Delta^{\boldsymbol{\alpha}})\leq \lambda_2(\Delta^{\boldsymbol{\alpha}})\leq \cdots \nearrow\infty.
\end{equation}
Due to (\ref{eq:quadratic form}), the corresponding Rayleigh quotient
of a smooth function $f:M\to \mathbb{C}$ is given by
\begin{equation}
  \mathcal{R}^{\boldsymbol{\alpha}}(f):=\frac{\int_M|(d+i\boldsymbol{\alpha})f|^2dx}{\int_M|f|^2dx}.
\end{equation}
The min-max principle (\ref{eq:minmax}) still holds in this
setting. In particular, we have
\begin{equation}
  \lambda_1(\Delta^{\boldsymbol{\alpha}})=\inf_{\substack{f\in C^{\infty}(M,\mathbb{C})\\\text{s.t. } f\not\equiv 0}} \mathcal{R}^{\boldsymbol{\alpha}}(f),
\end{equation}
where $C^{\infty}(M,\mathbb{C})$ is the set of smooth complex valued
functions.

Consider $U(1)$ as a subset $\{z\in \mathbb{C}\mid |z|=1\}$ of
$\mathbb{C}$ and denote the set of smooth maps from $M$ to $U(1)$
by $C^{\infty}(M, U(1))$. For $\tau \in C^{\infty}(M, U(1))$, we
then define by
\begin{equation}
  \boldsymbol{\alpha}_{\tau}:=\frac{d\tau}{i\tau}
\end{equation}
a smooth $1$-form. The set $\mathfrak{B}:=\{ \boldsymbol{\alpha}_{\tau}\mid \tau\in C^{\infty}(M, U(1))\}$
has the following characterization due to Shigekawa, \cite[Proposition 3.1 and Theorem 4.2]{Shigekawa87}. Since $a\boldsymbol{\alpha}_{\tau}=\boldsymbol{\alpha}_{\tau^a}$ for $a\in \mathbb{R}$ and $\boldsymbol{\alpha}_{\tau}+\boldsymbol{\alpha}_{\tau'}=\boldsymbol{\alpha}_{\tau\tau'}$, $\mathfrak{B}$ is in fact a real vector space.

\begin{Thm}[Shigekawa]\label{thm:Shigekawa}
  The following statements are equivalent:
  \begin{itemize}
  \item [(i)] $\lambda_1(\Delta^{\boldsymbol{\alpha}})=0$;
  \item [(ii)] $\boldsymbol{\alpha}\in \mathfrak{B}$;
  \item [(iii)] $d\boldsymbol{\alpha}=0$ and
    $\int_C\boldsymbol{\alpha}=0\mod 2\pi$, for any closed curve $C$
    in $M$.
  \end{itemize}
\end{Thm}

This result can be compared with Corollary \ref{cor:Cheeger constamt}: the set $\mathfrak{B}$
is comparable to the set of balanced signatures in the discrete setting. Locally, we can find a
smooth real-valued function $\theta$ such that $\tau=e^{i\theta}$ and $\boldsymbol{\alpha}_{\tau}=d\theta$.

In the discrete setting, Laplacians~$\Delta_\mu^s$ with switching equivalent signatures are unitarily equivalent
by (\ref{unitary_equivalence_of_laplace}) while magnetic Laplacians $\Delta^{\boldsymbol{\alpha}}$
are unitarily equivalent under \emph{gauge transformations} in the smooth setting. Recall that a
gauge transformation
\begin{equation}\label{eq:gaugeTrans}
 \boldsymbol{\alpha}\mapsto \boldsymbol{\alpha}+\boldsymbol{\alpha}_{\tau}
\end{equation}
is associated to any $\tau\in C^{\infty}(M, U(1))$. We have (\cite[Proposition 3.2]{Shigekawa87})
\begin{equation}
  \overline{\tau}\Delta^{\boldsymbol{\alpha}}\tau=\Delta^{\boldsymbol{\alpha}+\boldsymbol{\alpha}_{\tau}}.
\end{equation}
In particular, if $\boldsymbol{\alpha}\in \mathfrak{B}$,
then $\Delta^{\boldsymbol{\alpha}}$ is unitarily equivalent to $\Delta$. In
other words, $\mathfrak{B}$ is the set of magnetic potentials which
``can be gauged away''.

\begin{Def}\label{def:frustration manifold}
  Let $\boldsymbol{\alpha}$ be a magnetic potential on $M$. For any
  nonempty Borel subset $\Omega\subseteq M$, the \emph{frustration
    index} $\iota^{\boldsymbol{\alpha}}(\Omega)$ of $\Omega$ is
  defined as
  \begin{equation}
    \iota^{\boldsymbol{\alpha}}(\Omega) = \inf_{\tau\in C^{\infty}(\Omega, U(1))}\int_{\Omega}|(d+i\boldsymbol{\alpha})\tau|dx = \inf_{\boldsymbol{\eta}\in
      \mathfrak{B}_{\Omega}}\int_{\Omega}|\boldsymbol{\eta}+\boldsymbol{\alpha}|dx,\label{eq:frustration
      index manifold}
  \end{equation}
  where $\mathfrak{B}_{\Omega}:=\{\boldsymbol{\alpha}_{\tau}|\tau\in
  C^{\infty}(\Omega, U(1))\}$.
\end{Def}

Clearly, the frustration index $\iota^{\boldsymbol{\alpha}}(\Omega)$
is invariant under gauge transformations of the potential
$\boldsymbol{\alpha}$. Roughly speaking, the frustration index measures how far the potential $\boldsymbol{\alpha}$ is from the set $\mathfrak{B}_{\Omega}$.

For any Borel subset $\Omega\subseteq M$, we denote by
$\mathrm{vol}(\Omega)$ its Riemannian volume. Its boundary measure
$\mathrm{area}(\partial\Omega)$ is defined as
\begin{equation}
  \mathrm{area}(\partial\Omega):=\liminf_{r\to 0}\frac{\mathrm{vol}(\Omega_r)-\mathrm{vol}(\Omega)}{r},
\end{equation}
where $\Omega_r$ is the open $r$-neighborhood of $\Omega$. Let us
denote
\begin{equation}
  \phi^{\boldsymbol{\alpha}}(\Omega):=\frac{\iota^{\boldsymbol{\alpha}}(\Omega)+\mathrm{area}(\partial \Omega)}{\mathrm{vol}(\Omega)}.
\end{equation}

\begin{Def}\label{def:Cheeger constant manifold}
  Let $M$ be a closed Riemannian manifold with a magnetic potential~$\boldsymbol{\alpha}$. The
  \emph{$n$-way Cheeger constant} $h_n^{\boldsymbol{\alpha}}$ is
  defined as
  \begin{equation}
    h_n^{\boldsymbol{\alpha}}:=\inf_{\{\Omega_p\}_{[n]}}\ \max_{p\in [n]}\ \phi^{\boldsymbol{\alpha}}(\Omega_p),
  \end{equation}
  where the infimum is taken over all $n$-subpartitions $\{\Omega_p\}_{[n]}$
  with $\mathrm{vol}(\Omega_p)>0$ for every $p\in [n]$.
\end{Def}

In particular, the Cheeger constant $h_1^{\boldsymbol{\alpha}}$
vanishes if and only if $\boldsymbol{\alpha}\in \mathfrak{B}$. We
prove the following lower bound for the first eigenvalue
$\lambda_1(\Delta^{\boldsymbol{\alpha}})$.

\begin{Thm}\label{thm:CheegerManifold}
  Let $\boldsymbol{\alpha}$ be a magnetic potential on a closed connected
  Riemannian manifold $M$. Then we have
  \begin{equation}\label{eq:Cheeger inequality manifold}
    h_1^{\boldsymbol{\alpha}}\leq 2\sqrt{2\lambda_1(\Delta^{\boldsymbol{\alpha}})}.
  \end{equation}
\end{Thm}

We first prove the following Lemma which is an analogue of Lemma~\ref{lemma:coarea}.

\begin{Lem}[Coarea inequality]\label{lemma:coarea manifold}
  Let $\boldsymbol{\alpha}$ be a magnetic potential on $M$. For any
  nonzero smooth function $f: M\to \mathbb{C}$, we have
  \begin{equation}\label{eq:coarea inequality manifold}
    \int_0^{\infty}\left(\iota^{\boldsymbol{\alpha}}(\Omega^{f}(\sqrt{t}))+\mathrm{area}(\partial \Omega^f(\sqrt{t}))\right)dt\leq 2\sqrt{2}\int_M|f|\cdot|(d+i\boldsymbol{\alpha})f|dx,
  \end{equation}
  where we use the notation $\Omega^f(\sqrt{t}):=\{x\in
  M\mid \sqrt{t} \leq |f(x)|\}$.
\end{Lem}

\begin{proof}
  For convenience, we denote $f_0:=|f|$. W.l.o.g., we assume that
  $f_0(x)>0$, for any $x\in M$. Otherwise, we first consider
  integration over $\Omega^f(\varepsilon)$ in the right hand side of
  (\ref{eq:coarea inequality manifold}), $\varepsilon>0$, and then let
  $\varepsilon\to 0$.

  For the function $f$, we have the following associated $1$-form in
  $\mathfrak{B}$:
  \begin{equation}
    \boldsymbol{\eta}_f := \boldsymbol{\alpha}_{\frac{f}{f_0}}.
  \end{equation}
  Locally, there is a smooth real-valued function $\theta$ such that
  $f/f_0=e^{i\theta}$ and $\boldsymbol{\eta}_f=d\theta$. Therefore, we
  have locally
  \begin{equation}
    |(d+i\boldsymbol{\alpha})f|= |(d+i\boldsymbol{\alpha})(f_0e^{i\theta})|=|df_0+if_0(d\theta+\boldsymbol{\alpha})|.
  \end{equation}
  This implies that
  \begin{equation}
    |(d+i\boldsymbol{\alpha})f|=|df_0+if_0(\boldsymbol{\eta}_f+\boldsymbol{\alpha})|.
  \end{equation}
  Note that both $df_0$ and
  $f_0(\boldsymbol{\eta}_f+\boldsymbol{\alpha})$ are real-valued
  $1$-forms. We estimate
  \begin{equation}\label{eq:coarea Est1}
    |(d+i\boldsymbol{\alpha})f|=\sqrt{|df_0|^2+|f_0(\boldsymbol{\eta}_f+\boldsymbol{\alpha})|^2}\geq\frac{1}{\sqrt{2}}\left(|df_0|+|f_0(\boldsymbol{\eta}_f+\boldsymbol{\alpha})|\right).
  \end{equation}
  By the co-area formula, we have
  \begin{equation}\label{eq:coarea Est2}
    \int_M\, f_0\, |df_0|\, dx=\int_0^{\infty}t\cdot\mathrm{area}(\partial \Omega^{f_0}(t))\, dt.
  \end{equation}
  We also have
  \begin{align}\label{eq:coarea Est3}
    \int_Mf_0^2\, |\boldsymbol{\eta}_f+\boldsymbol{\alpha}|\, dx
		&= 2\int_0^{\infty} t\int_{\Omega^{f_0}(t)}|\boldsymbol{\eta}_f+\boldsymbol{\alpha}|\, dx\, dt\notag\\
		&\geq \int_0^{\infty} t\int_{\Omega^{f_0}(t)}|\boldsymbol{\eta}_f+\boldsymbol{\alpha}|\, dx\, dt.
  \end{align}
  Combining (\ref{eq:coarea Est1}), (\ref{eq:coarea Est2}), and (\ref{eq:coarea Est3}), we obtain
  \begin{align*}
    \int_M|f|\cdot|(d+i\boldsymbol{\alpha})f|dx\geq&\frac{1}{2\sqrt{2}}\int_0^{\infty}2t\left(\mathrm{area}(\partial \Omega^f(t))+\int_{\Omega^f(t)}|\boldsymbol{\eta}_f+\boldsymbol{\alpha}|dx\right)dt\\
    =&\frac{1}{2\sqrt{2}}\int_0^{\infty}\left(\mathrm{area}(\partial
      \Omega^f(\sqrt{t}))+\int_{\Omega^f(\sqrt{t})}|\boldsymbol{\eta}_f+\boldsymbol{\alpha}|dx\right)dt
  \end{align*}
  Recalling the definition of the frustration index
  (\ref{eq:frustration index manifold}), this proves the lemma.
\end{proof}

Similarly as in Section \ref{section:Cheeger inequality} for the discrete setting, we derive the following lemma from the coarea inequality, which is the continuous analogue of Lemma \ref{lemma:clusteringIU(1)}.
\begin{Lem}\label{lemma:prepare for higher}
Let $\boldsymbol{\alpha}$ be a magnetic potential on $M$. For any
  nonzero smooth function $f: M\to \mathbb{C}$, there exists $t'\in [0,\max_{x\in M}|f(x)|^2]$ such that
  \begin{equation}\label{eq:phiRayleigh}
     \phi^{\boldsymbol{\alpha}}(\Omega^f(\sqrt{t'}))\leq 2\sqrt{2\mathcal{R}^{\boldsymbol{\alpha}}(f)}.
  \end{equation}
\end{Lem}
\begin{proof}
First observe that there exists $t'$ such that
\begin{equation}
  \phi^{\boldsymbol{\alpha}}(\Omega^f(\sqrt{t'}))\leq \frac{\int_0^{\infty}\left(\iota^{\boldsymbol{\alpha}}(\Omega^{f}(\sqrt{t}))+\mathrm{area}(\partial \Omega^f(\sqrt{t}))\right)dt}{\int_0^\infty\mathrm{vol}(\Omega^f(\sqrt{t}))dt}.
\end{equation}
Note that $\int_M|f(x)|^2dx=\int_0^\infty\mathrm{vol}(\Omega^f(\sqrt{t}))dt$.
Then the lemma follows from applying the coarea inequality and Cauchy-Schwarz inequality.
\end{proof}

Theorem \ref{thm:CheegerManifold} is proved by applying Lemma \ref{lemma:prepare for higher} to the corresponding eigenfunction of $\lambda_1(\Delta^{\boldsymbol{\alpha}})$.
We also have the following higher order Cheeger inequalities for the
magnetic Laplacian $\Delta^{\boldsymbol{\alpha}}$.

\begin{Thm}\label{thm:higherManifold}
There exists an absolute constant $C>0$ such that for any closed connected Riemannian manifold
	$M$ with a magnetic
	potential~$\boldsymbol{\alpha}$ and $n \in {\mathbb N}$, we have
  \begin{equation}
    h_n^{\boldsymbol{\alpha}}\leq Cn^3\sqrt{\lambda_n(\Delta^{\boldsymbol{\alpha}})}.
  \end{equation}
\end{Thm}

For the proof, first consider Lemma \ref{lemma:localization manifold} below which is
an analogue of Lemma~\ref{lemma:localization}.  Let $F: M \to {\mathbb C}$ be the map given
by
\begin{equation}
F(x)=(f_1(x), f_2(x), \ldots, f_n(x))\in \mathbb{C}^n,
\end{equation}
where $f_i$ are orthonormal eigenfunctions that correspond to the eigenvalues
$\lambda_i(\Delta^{\boldsymbol{\alpha}})$ for $i \in [n]$. The
pseudometric $d_F$ on $M_F:=\{x\in M \mid F(x)\neq 0\}$ is defined by (\ref{eq:pseudometric}) via
\begin{equation}\label{eq:metricManifoldcase}
  d_F(x,y):=\inf_{\gamma\in U(1)}\left\Vert\frac{F(x)}{\Vert F(x)\Vert}-\gamma\frac{F(y)}{\Vert F(y)\Vert}\right\Vert.
\end{equation}
For $\epsilon > 0$, the cut-off function $\eta$ from (\ref{eq:cutoff}) is directly transferred to the manifold setting
and yields a localized function $\eta F$.

\begin{Lem}\label{lemma:localization manifold}
  For almost every $x\in M$, we have
  \begin{equation}\label{eq:localization manifold}
    \Vert (d+i\boldsymbol{\alpha})(\eta F)(x)\Vert^2\leq 2\left(1+\frac{4}{\epsilon^2}\right)\Vert (d+i\boldsymbol{\alpha})F(x)\Vert^2.
  \end{equation}
\end{Lem}

\begin{proof}
  If $F(x)=0$, the estimate (\ref{eq:localization manifold}) follows
  directly from $|\eta|\leq 1$. We therefore assume $F(x)\neq 0$
  in the following and set $f_{p,0}:=|f_p|$ for every $p\in [n]$. Then there
  is a real-valued function $\theta_p$ that is defined in a small neighborhood
  of $x\in M$ such that $f_p=f_{p,0}e^{i\theta_p}$. We now obtain at~$x$
  \begin{align}
    \Vert (d+i\boldsymbol{\alpha})(\eta F)\Vert^2
		=& \ \sum_{p\in [n]}
				|(d+i\boldsymbol{\alpha})(\eta f_{p,0}e^{i\theta_p})|^2 \notag\\
    	=& \ \sum_{p\in [n]}
				|f_{p,0}\; d\eta+\eta df_{p,0}+i(\eta f_{p,0})(\boldsymbol{\alpha}+d\theta_p)|^2 \notag\\
     \leq& \ \sum_{p\in [n]}
				\left( 2f_{p,0}^2|d\eta|^2+2|\eta|^2|df_{p,0}|^2+|f_{p,0}(\boldsymbol{\alpha}+d\theta_p)|^2 \right) \notag\\
    \leq & \ 2 |d\eta|^2 \sum_{p\in [n]} f_{p,0}^2
			 + 2 \sum_{p\in [n]} |df_{p,0}+if_{p,0}(\boldsymbol{\alpha}+d\theta_p)|^2 \notag\\
    	=& \ 2|d\eta|^2 \Vert F\Vert^2+2\Vert(d+i\boldsymbol{\alpha})F\Vert^2.
		\label{eq:loc waitforfurther}
  \end{align}
  There exist a unit tangent vector $\sigma'(0)\in T_xM$ such that
  \begin{equation}
    |d\eta(x)|=\lim_{t\to 0}\frac{|\eta(\sigma(t))-\eta(\sigma(0))|}{t},
  \end{equation}
  where $\sigma(t):=\mathrm{exp}_x(t\sigma'(0))$ is the geodesic with
  $\sigma(0)=x$. Since we have
  \begin{equation}
    |\eta(\sigma(t))-\eta(\sigma(0))|\leq
    \frac{1}{\epsilon}\cdot d_F(\sigma(t), \sigma(0)),
  \end{equation}
  we conclude
  \begin{equation}\label{eq:detaF}
    |d\eta(x)|\cdot \Vert F(x)\Vert\leq \frac{1}{\epsilon}\cdot\lim_{t\to 0}\frac{d_{F}(\sigma(t), \sigma(0))\cdot \Vert F(x)\Vert}{t}.
  \end{equation}
  Using (\ref{eq:metricManifoldcase}) and setting
  \begin{equation}
  \gamma(t):=e^{i\int_0^t\langle
    \boldsymbol{\alpha}(\sigma(t)), \sigma'(t)\rangle dt},
  \end{equation}
  we obtain
  \begin{align}
    d_{F}(\sigma(t), \sigma(0))\Vert F(x)\Vert
		\leq& \ \left\Vert \gamma(t)\frac{F(\sigma(t))}{\Vert F(\sigma(t))\Vert}-\frac{F(\sigma(0))}{\Vert F(\sigma(0))\Vert} \right\Vert
				\cdot \Vert F(x)\Vert \notag\\
    =& \ \left\Vert\frac{G(t)}{\Vert G(t)\Vert}-\frac{G(0)}{\Vert G(0)\Vert} \right\Vert
				\cdot \Vert G(0)\Vert,\label{eq:dFF}
  \end{align}
  where $G(t):=\gamma(t)F(\sigma(t))$. Now we can carry out similar
  estimates as in Lemma~\ref{lemma:keyForLocal}. Although we do not
  know whether $\Vert G(0)\Vert$ is smaller than $\Vert G(t)\Vert$, we
  still obtain
  \begin{align}
    \left\Vert\frac{G(t)}{\Vert G(t)\Vert}-\frac{G(0)}{\Vert G(0)\Vert}\right\Vert\Vert G(0)\Vert
		\leq & \ \left\Vert\frac{\Vert G(0)\Vert}{\Vert G(t)\Vert}\cdot G(t)-G(t)\right\Vert+\Vert G(t)-G(0) \Vert\notag\\
    	\leq & \ 2\cdot\Vert G(t)-G(0)\Vert.\label{eq:Gt}
  \end{align}
  Inserting (\ref{eq:dFF}) and (\ref{eq:Gt}) into (\ref{eq:detaF}), we
  obtain
  \begin{align}
    |d\eta(x)\, |\cdot\Vert F(x)\Vert
		\leq& \ \frac{2}{\epsilon} \cdot \lim_{t\to 0}\frac{\Vert G(t)-G(0)\Vert}{t} \notag\\
    	   =& \ \frac{2}{\epsilon} \cdot \lim_{t\to 0}\frac{\sqrt{\sum_{p\in [n]} |\gamma(t)f_p(\sigma(t))-\gamma(0)f_p(\sigma(0))|^2}}{t} \notag\\
    	   =& \ \frac{2}{\epsilon} \cdot \sqrt{\sum_{p\in [n]} \left|\lim_{t\to 0}\frac{\gamma(t)f_p(\sigma(t))-\gamma(0)f_p(\sigma(0))}{t}\right|^2} \notag\\
    	   =& \ \frac{2}{\epsilon} \cdot \sqrt{\sum_{p\in [n]} \left|\langle (d+i\boldsymbol{\alpha})f_p(x), \sigma'(0) \rangle\right|^2}.
  \end{align}
  In the last equality above, we used the fact that
  $\frac{d}{dt}|_{t=0}\gamma(t)=i\langle\boldsymbol{\alpha}(x),\sigma'(0)\rangle$. Since
  $|\sigma'(0)|=1$, we conclude
  \begin{equation}\label{eq:locmanifold final}
    |d\eta(x)|\cdot \Vert F(x)\Vert\leq \frac{2}{\epsilon}\Vert(d+i\boldsymbol{\alpha})F(x)\Vert.
  \end{equation}
  Combining (\ref{eq:locmanifold final}) and (\ref{eq:loc
    waitforfurther}), we finally obtain (\ref{eq:localization
    manifold}).
\end{proof}

Note that the pseudometric (\ref{eq:metricManifoldcase}) induced from
the metric on a complex projective space played an important role in
the proof.
\begin{proof}[Proof of Theorem \ref{thm:higherManifold}] Applying Theorem \ref{thm:decomposition} to $(M_F, d_F, \Vert F(x)\Vert^2dx)$, we obtain a subpartition $\{T_i\}_{[n]}$ of $M_F$, such that
\begin{itemize}
  \item [(i)] $d_F(T_p, T_q)\geq \frac{2}{C_0n^{5/2}}$, for all $p,q\in [n]$, $p\neq q$,
  \item [(ii)] $\int_{T_p}\Vert F(x)\Vert^2dx\geq \frac{1}{2n}\int_M\Vert F(x)\Vert^2dx$, for all $p\in [n]$,
  \end{itemize}
where $C_0$ is an absolute constant.
Employing further Lemma \ref{lemma:prepare for higher} and Lemma \ref{lemma:localization manifold}, the proof of the theorem can be done via the same arguments as in Section
\ref{section:proofof higher}.
\end{proof}


\section*{Acknowledgements}
We like to express our gratitude to Afonso S. Bandeira for pointing out
the relation between magnetic and connection Laplacians and useful references.
SL is very grateful to Alexander Grigor'yan for inspiring discussions about decompositions of spaces.
CL, SL and NP acknowledge the support of the EPSRC Grant EP/K016687/1
``Topology, Geometry and Laplacians of Simplicial Complexes''. CL also
acknow\-ledges the support of the SFB TRR109 ``Discretization in
Geometry and Dynamics'', the kind hospitality of the Department of
Mathematical Sciences of Durham University and of the Grey College.


\begin{thebibliography}{99}
\bibitem{Alon1986} N. Alon, Eigenvalues and expanders, Combinatorica 6
  (1986), no. 2, 83-96.
\bibitem{AM1985} N. Alon, V. Milman, $\lambda_1$, isoperimetric inequalities for graphs, and superconcentrators, J. Combin. Theory Ser. B 38 (1985), no. 1, 73-88.
\bibitem{AtayLiu14} F. M. Atay, S. Liu, Cheeger constants, structural balance, and spectral clustering analysis for signed graphs, arXiv: 1411.3530, November 2014.
\bibitem{BSS13} A. S. Bandeira, A. Singer, D. A. Spielman, A Cheeger inequality for the graph connection Laplacian,
SIAM J. Matrix Anal. Appl. 34 (2013), no. 4, 1611-1630.
\bibitem{BJ} F. Bauer, J. Jost, Bipartite and neighborhood graphs and the spectrum of the normalized graph Laplacian, Comm. Anal. Geom. 21 (2013), no. 4, 787-845.
\bibitem{BKW12} F. Bauer, M. Keller, R. K. Wojciechowski, Cheeger inequalities for unbounded graph Laplacians, J. Eur. Math. Soc. (JEMS) 17 (2015), no. 2, 259-271.
\bibitem{Cheeger1970} J. Cheeger, A lower bound for the smallest eigenvalue of the Laplacian, Problems in analysis (Papers dedicated to Salomon Bochner, 1969), pp. 195-199. Princeton Univ. Press, Princeton, N. J., 1970.
\bibitem{CoiWei1971} R. Coifman, G. Weiss, Analyse harmonique non-commutative sur certains espaces homog{\`e}nes, vol. 242 of Lecture Notes in
    Mathematics, Springer, Berlin-New York, 1971.
\bibitem{ColindeVTT11}Y. Colin de Verdi\`{e}re, N. Torki-Hamza, F. Truc, Essential self-adjointness for combinatorial Schr\"{o}dinger operators III-Magnetic fields, Ann. Fac. Sci. Toulouse Math. (6) 20 (2011), no. 3, 599-611.
\bibitem{DR1994} M. Desai, V. Rao. A characterization of the smallest eigenvalue of a graph, J. Graph Theory 18 (1994), no. 2, 181-194.
\bibitem{Dodziuk1984} J. Dodziuk, Difference equations, isoperimetric inequality and transience of certain random walks, Trans. Amer. Math. Soc. 284 (1984), no. 2, 787-794.
\bibitem{DodziukMathai06} J. Dodziuk, V. Mathai, Kato's inequality and asymptotic spectral properties for discrete magnetic Laplacians, The ubiquitous heat kernel, 69-81, Contemp. Math., 398, Amer. Math. Soc., Providence, RI, 2006.
\bibitem{Erdos96} L. Erd\H{o}s, Rayleigh-type isoperimetric inequality with a homogeneous magnetic field, Calc. Var. 4 (1996), 283-292.
\bibitem{FLM08} R. L. Frank, A. Laptev, S. Molchanov,
Eigenvalue estimates for magnetic Schr\"{o}dinger operators in domains,
Proc. Amer. Math. Soc. 136 (2008), no. 12, 4245-4255.
\bibitem{Funano2013} K. Funano, Eigenvalues of Laplacian and multi-way isoperimetric constants on weighted Riemannian manifolds, arXiv:1307.3919v1, July 2013.
\bibitem{FS2013} K. Funano, T. Shioya, Concentration, Ricci curvature, and eigenvalues of Laplacian, Geom. Funct. Anal. 23 (2013), no. 3, 888-936.
\bibitem{Golenia14} S. Gol\'{e}nia, Hardy inequality and asymptotic eigenvalue distribution for discrete Laplacians, J. Funct. Anal. 266 (2014), 2662-2688.
\bibitem{GrigoryanNetrusovYau} A. Grigor'yan, Y. Netrusov, S.-T, Yau,
Eigenvalues of elliptic operators and geometric applications. Surveys in differential geometry. Vol. IX, 147¨C217,
Surv. Differ. Geom., IX, Int. Press, Somerville, MA, 2004.
\bibitem{Gross74} J. L. Gross, Voltage graphs, Discrete Math. 9 (1974), 239-246.
\bibitem{GKL2003} A. Gupta, R. Krauthgamer, J. R. Lee, Bounded geometries, fractals, and low-distortion embeddings, 44th Symposium on Foundations of Computer Sciences, 534-543, 2003.
\bibitem{Harary53} F. Harary, On the notion of balance of a signed graph, Michigan Math. J. 2 (1953), no. 2, 143-146.
\bibitem{Harary59} F. Harary, On the measurement of structural balance, Behavioral Sci. 4 (1959), 316-323.
\bibitem{HararyPalmer66} F. Harary, E. Palmer, Enumeration of mixed graphs, Proc. Amer. Math. Soc. 17 1966 682-687.
\bibitem{HinzTeplyaev13} M. Hinz, A. Teplyaev, Dirac and magnetic Schr\"{o}dinger operators on fractals, J. Funct. Anal. 265 (2013), 2830-2854.
\bibitem{HLW06} S. Hoory, N. Linial, A. Wigderson, Expander graphs and their applications, Bull. Amer. Math. Soc. 43 (2006), no. 4, 439-561.
\bibitem{Jost05} J. Jost, Riemannian geometry and geometric analysis, Fourth edition, Universitext, Springer-Verlag, Berlin, 2005.
\bibitem{KLLGT2013} T.-C. Kwok, L.-C. Lau, Y.-T. Lee, S. Oveis Gharan, L. Trevisan, Improved Cheeger's inequality: Analysis of spectral partitioning algorithms through higher order spectral gap, STOC'13-Proceedings of the 2013 ACM Symposium on Theory of Computing, 11-20, ACM, New York, 2013.
\bibitem{LOT2013} J. R. Lee, S. Oveis Gharan, L. Trevisan, Multi-way spectral partitioning and higher-order Cheeger inequalities, STOC'12-Proceedings of the 2012 ACM Symposium on Theory of Computing, 1117-1130, ACM, New York, 2012; J. ACM 61 (2014), no. 6, 37:1-30.
\bibitem{LeeNaor2005} J. R. Lee, A. Naor, Extending Lipschitz functions via random metric partitions, Invent. Math. 160 (2005), no. 1, 59-95.
\bibitem{LiebLoss1993} E. Lieb, M. Loss, Fluxes, Laplacians, and Kasteleyn's theorem, Duke Math. J. 71 (1993), no. 2, 337-363.
\bibitem{Liu13} S. Liu, Multi-way dual Cheeger constants and spectral
  bounds of graphs, Adv. Math. 268 (2015), 306-338.
\bibitem{Liu14} S. Liu, An optimal dimension-free upper bound for eigenvalue ratios, arXiv: 1405.2213, May 2014.
\bibitem{LP14} S. Liu, N. Peyerimhoff, Eigenvalue ratios of nonnegatively curved graphs, arXiv: 1406.6617, June, 2014.
\bibitem{LPV14} S. Liu, N. Peyerimhoff, A. Vdovina, Signatures, lifts and eigenvalues of graphs, arXiv: 1412.6841, December 2014.
\bibitem{Luxburg07} U. von Luxburg, A tutorial on spectral clustering, Statistics and computing, 17 (2007), no. 4, 395-416.
\bibitem{MSS} A. W. Marcus, D. A. Spielman, N. Srivastava, Interlacing families I: bipartite Ramanujan graphs of all degrees, 2013 IEEE 54th Annual Symposium on Foundations of Computer Science - FOCS 2013, 529-537, IEEE Computer Soc., Los Alamitos, CA, 2013;  Ann. of Math. 182 (2015), 307-325.
\bibitem{Miclo2008} L. Miclo, On eigenfunctions of Markov processes on trees, Probab. Theory Related Fields 142 (2008), no. 3-4, 561-594.
\bibitem{Miclo2013} L. Miclo, On hyperboundedness and spectrum of Markov operators, Invent. Math. 200 (2015), no. 1, 311-343.
\bibitem{MT12} A. Morame, F. Truc,
Counting function of the embedded eigenvalues for some manifold with cusps, and magnetic Laplacian,
Math. Res. Lett. 19 (2012), no. 2, 417-429.
\bibitem{Nic87} S. Nicaise,
Spectre des r\'eseaux topologiques finis, Bull.\
Sci.\ Math.\ (2) 111 (1987), no.~4, 401-413.
\bibitem{Paternain01} G. P. Paternain,
Schr\"{o}dinger operators with magnetic fields and minimal action functionals,
Israel J. Math. 123 (2001), 1-27.
\bibitem{Po09} O. Post,
Spectral analysis of metric graphs and related spaces,
in ``Limits of graphs in group theory'', eds. G. Arzhantseva and A.
Valette, Presses Polytechniques et Universitaires Romandes, 109--140
(2009), 109-140.
\bibitem{Ries07} B. Ries, Coloring some classes of mixed graphs, Discrete Appl. Math. 155 (2007), no. 1, 1-6.
\bibitem{SadeghiLauritzen14} K. Sadeghi, S. Lauritzen,
Markov properties for mixed graphs,
Bernoulli 20 (2014), no. 2, 676-696.
\bibitem{Shigekawa87} I. Shigekawa, Eigenvalue problems for the Schr\"{o}dinger operator with the magnetic field on a compact Riemannian manifold,
J. Funct. Anal. 75 (1987), no. 1, 92-127.
\bibitem{Shubin94} M. A. Shubin, Discrete magnetic Laplacian, Comm. Math. Phys. 164 (1994), no. 2, 259-275.
\bibitem{Shubin01} M. A. Shubin,
Essential self-adjointness for semi-bounded magnetic Schr\"{o}dinger operators on non-compact manifolds,
J. Funct. Anal. 186 (2001), no. 1, 92-116.
\bibitem{SH1972} B. Simon, R. H{\o}egh-Krohn, Hypercontractive semigroups and two dimensional self-coupled Bose fields, J. Funct. Anal. 9 (1972), 121-180.
\bibitem{SingerWu} A. Singer, H.-T. Wu, Vector diffusion maps and the connection Laplacian,
Comm. Pure Appl. Math. 65 (2012), no. 8, 1067-1144.
\bibitem{Sotskov00} Y. N. Sotskov,
Scheduling via mixed graph coloring, Operations Research Proceedings 1999 (Magdeburg), 414-418, Springer, Berlin, 2000.
\bibitem{Sunada93} T. Sunada, A discrete analogue of periodic magnetic Schr\"{o}dinger operators, Geometry of the spectrum (Seattle, WA, 1993), 283-299, Contemp. Math., 173, Amer. Math. Soc., Providence, RI, 1994.
\bibitem{Trevisan2012} L. Trevisan, Max cut and the smallest eigenvalue, STOC'09-Proceedings of the 2009 ACM International Symposium on Theory of Computing, 263-271, ACM, New York, 2009;
SIAM J. Comput. 41 (2012), no. 6, 1769-1786.
\bibitem{VT77} J. Vannimenus, G. Toulouse, Theory of the frustration effect: II. Ising spins on a square lattice. J. Phys. C: Solid State Phys. 10 (1977), L537.
\bibitem{Wang2014} F.-Y. Wang, Criteria of spectral gap for Markov operators, J. Funct. Anal. 266 (2014), 2137-2152.
\bibitem{Zaslavsky82} T. Zaslavsky, Signed graphs, Discrete Appl. Math. 4 (1982), no. 1, 47-74.
\bibitem{ZaslavskyMatrices} T. Zaslavsky, Matrices in the theory of signed simple graphs, Advances in discrete mathematics and applications: Mysore, 2008, 207-229, Ramanujan Math. Soc. Lect. Notes Ser., 13, Ramanujan Math. Soc., Mysore, 2010.
\bibitem{Zaslavsky14} T. Zaslavsky, private communication.
\bibitem{ZhangLi02} X.-D. Zhang, J.-S. Li, The Laplacian spectrum of a mixed graph, Linear Algebra Appl. 353 (2002), 11-20.
\end{thebibliography}
\end{document}